\documentclass[11pt, leqno]{amsart}

\makeatletter

\def\subsection{\@startsection{subsection}{2}%
	{\parindent}{.5\linespacing}{-.5em}%
	{\normalfont\itshape\/}}
\def\subsubsection{\@startsection{subsubsection}{3}%
	{\parindent}{.5\linespacing}{-.5em}%
	{\normalfont\itshape\/}}
\def\appendix{\par\c@section\z@ \c@subsection\z@
	\gdef\theHsection{\Hy@AlphNoErr{section}}%
\let\sectionname\appendixname{}
\def\thesection{{\upshape\@Alph\c@section}}}

\newtheoremstyle{plain}{0.5\linespacing}{0.5\linespacing}{\itshape}%
	{\parindent}{\scshape}{.}{0.5em}%
	{\thmname{#1}\thmnumber{ #2}\thmnote{\normalfont{} (#3)}}
\newtheoremstyle{definition}{0.5\linespacing}{0.5\linespacing}%
	{\upshape}{\parindent}%
	{\itshape}{.}{0.5em} 
	{\thmname{#1}\thmnumber{ #2}\thmnote{\normalfont{} (#3)}}
\newtheoremstyle{remark}{0.5\linespacing}{0.5\linespacing}%
	{\upshape}{\parindent}%
	{\itshape}{.}{0.5em} 
	{\thmname{#1}\thmnumber{ #2}\thmnote{\normalfont{} (#3)}}

\renewenvironment{proof}[1][\proofname]{\par 
	\pushQED{\qed}%
	\normalfont\topsep6\p@\@plus6\p@\relax
	\trivlist%
	\item[\hskip\labelsep\hskip\parindent%
	\itshape%
	#1\@addpunct{.}]\ignorespaces%
	}{ 
	\popQED\endtrivlist\@endpefalse%
}
\makeatother
\allowdisplaybreaks

\usepackage{amsmath, amsthm, amssymb}
\usepackage{enumitem}
\usepackage{xcolor}
\definecolor{cite}{rgb}{0.30,0.60,1.00}
\definecolor{url}{rgb}{0.00,0.00,0.80}
\definecolor{link}{rgb}{0.40,0.10,0.20}

\usepackage[pdfusetitle,colorlinks,linkcolor=link,urlcolor=url,citecolor=cite,breaklinks,bookmarksdepth=3,bookmarksopen=true]{hyperref}
\usepackage[top=1.75in, headsep=0.25in, left=1.65in, right=1.7in,bottom=1.275in]{geometry}
\usepackage{url}
\usepackage{graphicx}
\usepackage{caption}
\usepackage{subcaption}
\usepackage[capitalise]{cleveref}
\usepackage{mathdots}
\usepackage{tikz-cd}
\usepackage{comment}
\usepackage{array}

\usepackage{mathtools}
\usepackage{physics}
\usepackage{float}
\usepackage[mathcal]{eucal}
\usepackage{microtype}
\usepackage{pagesel}

\newtheorem{theorem}{Theorem}[section]

\newtheorem{proposition}[theorem]{Proposition}
\newtheorem{lemma}[theorem]{Lemma}

\theoremstyle{definition}
\newtheorem{definition}[theorem]{Definition}
\theoremstyle{definition}

\theoremstyle{remark}
\newtheorem{remark}[theorem]{Remark}
\theoremstyle{remark}
\newtheorem{example}[theorem]{Example}

\newcommand{\nNatural}{\mathbb{N}}
\newcommand{\zIntegers}{\mathbb{Z}}

\newcommand{\rReal}{\mathbb{R}}
\newcommand{\rRealPos}{\rReal_{+}}
\newcommand{\cComplex}{\mathbb{C}}

\newcommand{\hHalfplane}{\mathbb{H}}
\newcommand{\idmap}{\mathrm{id}}
\renewcommand{\abs}[1]{\left|#1\right|}

\newcommand{\pullback}[1]{{#1}^{*}}
\newcommand{\pushforward}[1]{{#1}_{*}}
\newcommand{\comp}{\mathbin{\circ}}
\newcommand{\boundary}{\partial}
\let\setminusaux\setminus{}	
\renewcommand*{\setminus}{\scalemath{\mathrel}{0.7}{\setminusaux}}

\newcommand{\st}{\mid}
\newcommand{\restrict}[2]{{#1}|_{#2}}
\newcommand{\closure}[1]{\overline{#1}}
\newcommand{\conjugate}[1]{\overline{#1}}
\newcommand{\embeds}{\hookrightarrow}
\renewcommand{\complement}[1]{{#1}^c}

\renewcommand{\exp}{\operatorname{exp}}

\renewcommand{\leq}{\leqslant}
\renewcommand{\geq}{\geqslant}

\let\tmp\epsilon{}
\let\epsilon\varepsilon{}
\let\varepsilon\tmp{}

\let\tmp\phi{}
\let\phi\varphi{}
\let\varphi\tmp{}

\let\emptyset\varnothing

\newcommand{\structure}{\mathcal{S}}
\newcommand{\RrPfaff}{\rReal_{\mathrm{rPfaff}}}
\newcommand{\Rexp}{\rReal_{\exp}}
\newcommand{\Ran}{\rReal_{\mathrm{an}}}
\newcommand{\Ralg}{\rReal_{\mathrm{alg}}}

\newcommand{\s}{\raisebox{.5ex}{\scalebox{0.6}{\#}}}
\newcommand{\so}{\s\kern-.02em{}o}
\NewDocumentCommand{\format}{sO{F}}{
	\IfBooleanTF{#1}
		{#2, \ell}
		{#2}
	}
\newcommand{\degree}[1][D]{#1}
\NewDocumentCommand{\FD}{sooo}{
	\IfBooleanTF{#1}								
		{											
			\IfNoValueTF{#4}
			{\mathcal{S}_{\order[#2]{1},\poly_{#2}\!\qty(#3)}}
			{
				{\mathcal{S}^{#4}_{\order[#2]{1},\poly_{#2}\!\qty(#3)}}
			}
		}
		{											
		\IfNoValueTF{#2}
			{\mathcal{S}}
			{
				\IfNoValueTF{#4}
				{
					{\mathcal{S}_{#2,#3}}
				}
				{\mathcal{S}^{#4}_{#2,#3}}
				}		
		}
}
\NewDocumentCommand{\polyfd}{oo}{
	\IfValueTF{#1}
		{\poly_{\format[#1]}\qty(\degree[#2])}
		{\poly_{\format}\qty(\degree)}
}
\newcommand{\polyfld}{\polyfd[\format,\ell][\degree]}

\newcommand{\cell}[1][C]{\mathcal{#1}}
\newcommand{\hyperbolicParameter}[1]{\{#1\}}
\newcommand{\point}{*}
\newcommand{\disc}[1][\relax]{
	\ifx\relax#1 
		D
	\else
		D\qty(#1)
	\fi
}

\NewDocumentCommand{\puncDisc}{so}{
	\IfBooleanTF{#1}
		{D_{\infty}}
		{D_{\circ}}
	\IfNoValueF{#2}{
		\qty(#2)
		}
}

\NewDocumentCommand{\annulus}{oo}{
	\IfNoValueTF{#1} 	
		{A} 			
		{A\qty(#1,#2)} 	
}
\RenewDocumentCommand{\circle}{o}{
	\IfNoValueTF{#1}{S}{S\qty(#1)}	
}
\NewDocumentCommand{\skeleton}{m}{
	\operatorname{Skel}\qty({#1})
	}
\NewDocumentCommand{\ext}{smm}{
	\IfBooleanTF{#1}		
		{\qty(#2){}^{#3}} 	
		{#2^{#3}}			
}
\NewDocumentCommand{\hExt}{smm}{
	\IfBooleanTF{#1}			
		{\qty(#2){}^{\hyperbolicParameter{#3}}} 
		{#2^{\hyperbolicParameter{#3}}}			
}
\newcommand{\initial}[3]{#1_{#2..#3}}
\NewDocumentCommand{\nuCover}{mo}{
	\IfNoValueTF{#2}
		{{#1}_{\times\nu}}
		{{#1}_{\times#2}}
}

\newcommand{\projbase}{\initial{\pi}{1}{\ell}}
\newcommand{\projfiber}{\pi_{\ell+1}}

\newcommand{\fundamentalGroup}{\pi_1}
\newcommand{\poly}{\operatorname{poly}}
\newcommand{\polyl}{\poly_{\ell}}

\newcommand{\varZ}{\mathbf{z}}
\newcommand{\varW}{\mathbf{w}}
\newcommand{\varZeta}{\boldsymbol{\zeta}}
\NewDocumentCommand{\polydisc}{soo}{
	\IfBooleanTF{#1}					
		{\Delta_{#2}\times\Delta_{#3}}	
		{\IfNoValueTF{#2}				
			{\Delta}
			{\Delta_{#2}}
		}
}
\RenewDocumentCommand{\order}{som}{
	\IfBooleanTF{#1}
		{
		\IfNoValueTF{#2}
			{\Omega\qty(#3)}
			{\Omega_{#2}\qty(#3)}
		}	
		{
		\IfNoValueTF{#2}
			{O\qty(#3)}
			{O_{#2}\qty(#3)}
		}	
}
\NewDocumentCommand{\ball}{ooo}{
	B
	\IfValueT{#1}
		{
			\qty(\IfValueTF{#2}
					{#1, #2}
					{#1}
				\IfValueT{#3}
					{; #3})
		}
}
\NewDocumentCommand{\diam}{mo}
	{
		\operatorname{diam}\qty(#1
		\IfValueT{#2}
			{; #2})
	}
\NewDocumentCommand{\dist}{mmo}
	{
		\operatorname{dist}\qty(#1,#2
		\IfValueT{#3}
			{\,; #3})
	}
\NewDocumentCommand{\latticeComplement}{o}{
	\IfNoValueTF{#1}
		{\cComplex\setminus\zIntegers^2}
		{\cComplex\setminus{\! #1}\zIntegers^2}
}

\newcommand{\Rouche}{Rouch\'e}

\makeatletter
\newdimen\scalemath@axis{}
\newcommand*{\scalemath}[3]{%
  #1{%
    \mathpalette{\scalemath@aux{#2}}{#3}%
  }%
}
\newcommand*{\scalemath@aux}[3]{%
  \begingroup
    \everyvbox{}%
    \settoheight\scalemath@axis{$#2\vcenter{}$}%
    \raisebox{\scalemath@axis}{%
      \scalebox{#1}{%
        \raisebox{-\scalemath@axis}{%
          $\m@th#2#3$%
        }%
      }%
    }%
  \endgroup
}
\makeatother 

\hypersetup{pdfauthor={Oded Carmon}}

\title{Complex Cells in Sharply O-minimal Structures}

\author{Gal Binyamini}
\address{The Weizmann Institute of Science, Rehovot, Israel}
\email{gal.binyamini@weizmann.ac.il}

\author{Oded Carmon}
\address{The Weizmann Institute of Science, Rehovot, Israel}
\email{oded.carmon@weizmann.ac.il}

\author{Dmitry Novikov}
\address{The Weizmann Institute of Science, Rehovot, Israel}
\email{dmitry.novikov@weizmann.ac.il}

\date{\today}

\begin{document}

\begin{abstract}
	We extend the theory of complex cells introduced by Binyamini and Novikov to the sharply o-minimal setting, obtaining cellular preparation and parameterization theorems which are polynomially effective in the degrees of the relevant sets.
	Our constructions are definable, and so applying them to sets in a given reduct of $\Ran$ yields cells and cellular maps definable in the same reduct.
\end{abstract}

\maketitle

{\small \tableofcontents}

\addtocontents{toc}{\protect\setcounter{tocdepth}{1}}

\section{Introduction}\label{sec:introduction}
Complex cells, introduced in~\cite{BinyaminiNovikov2019}, are holomorphic analogs of the cells of o-minimal geometry (e.g.\ semialgebraic or subanalytic geometry).
A classical o-minimal cell is, roughly, a ``product'' $I_1\odot\cdots\odot I_n$ of intervals $\{I_j\}$, where the fibers are translated and scaled by continuous functions on their respective bases.
A complex cell $\cell[F]_1\odot\cdots\odot\cell[F]_n$ is composed in a similar manner from basic fibers $\{\cell[F]_i\}$ such as discs and annuli in the complex plane, where the radii of the fibers are determined by holomorphic functions of the relevant bases (see \Cref{sec: complex cells}).

In~\cite{BinyaminiNovikov2019}, Binyamini and Novikov develop the theory of complex cells in the context of the structures $\Ralg$ and $\Ran$.
They obtain uniform parameterization and preparation theorems for families of semialgebraic or globally subanalytic sets and functions.
In the case where the relevant sets or functions are semialgebraic, their results are effective and depend polynomially on the degrees of the polynomials defining the given sets.

The constructions in the algebraic and analytic settings differ in several points, with the following consequences.
First, in the case of sets which are, for example, sub-Pfaffian but not necessarily semialgebraic, the arguments of~\cite{BinyaminiNovikov2019} give parameterizations whose size and complexity might be exponential in the degrees of the parameterized sets.
Second, it is not clear that applying the constructions in the more general analytic setting to sets definable in some reduct of $\Ran$ yields cells definable in the same reduct --- these constructions depend on the definability of coefficients in Laurent expansions of definable functions.

In this paper, we extend the results of~\cite{BinyaminiNovikov2019} to the setting of a sharply o-minimal expansion of the real field, as defined by Binyamini, Novikov, and Zak in~\cite{BinyaminiNovikovZak2022} (see \Cref{sec:so minimal}).
An example is the structure $\RrPfaff$ generated by restricted Pfaffian functions.

We obtain effective preparation and parameterization theorems (\Cref{thm:sharp preparation,thm:cpt}), which are uniform in families, and which depend polynomially on the degrees of the relevant sets and functions.
This resolves~\cite[Conjecture 21]{BinyaminiNovikov2023}.
As an added benefit, our constructions are definable in the sense that applying them in any reduct of $\Ran$ (ignoring questions of effectivity) yields uniformly finite covers which are definable in the same reduct.
The precise statements require some technical set-up which we review in \Cref{sec:so minimal,sec: complex cells}.
The following are informal versions one may keep in mind.

{
	\renewcommand{\thetheorem}{A}
	\begin{theorem}[Sharp Cellular Parametrization Theorem, \s CPT]
		\label{thm: intro cpt}
		Let $\cell=\cell[F]_1\odot\cdots\odot\cell[F]_\ell\subset\cComplex^\ell$ be a complex cell and let $Z_1,\dots,Z_k$ be a finite collection of analytic hypersurfaces of $\cell$, all of them definable in an o-minimal structure~$\structure$.
		Then there is a finite collection of complex cells $\cell_i$ and holomorphic maps $f_i:\cell_i\to\cComplex^\ell$, all of them definable in $\structure$, such that the sets $f_i(\cell_i)\cap\cell$ cover $\cell$ and each of them is either disjoint from or contained in each of the hypersurfaces~$Z_j$.

		Assume now that $\structure$ is sharply o-minimal with FD-filtration $\{\FD[\format][\degree]\}$.
		If the hypersurfaces $\{Z_j\}$ as well as the radii functions determining the fibers $\{\cell[F]_j\}$ are in $\FD[\format][\degree]$, then the number of maps $\{f_i\}$ may be taken to be $\polyfd[\format][\degree,k]$, and the radii functions determining the fibers of the cells $\cell_i$ may be taken to be in $\FD*[\format][\degree]$.
	\end{theorem}

	\renewcommand{\thetheorem}{B}
	\begin{theorem}[Sharp Cellular Preparation Theorem, \s CPrT]
		\label{thm: intro cprt}
		The holomorphic maps $f_i:\cell_i\to\cComplex^\ell$ in the previous theorem may be taken to be of the form
		\begin{equation}
			(f_i)_j(z_1,\dots,z_n)=\pm z_j^{q_{i,j}}+\phi_{i,j}(z_1,\dots,z_{j-1}),\quad 1\leq j\leq \ell,
		\end{equation}
		where $q_{i,j}\in\zIntegers\setminus\{0\}$ and $\phi_{i,j}$ is holomorphic.
		That is the maps $f_i$ are given, up to signs, by the composition of power maps and holomorphic translations which depend only on previous variables.
	\end{theorem}
}

We note that, similarly to~\cite{BinyaminiNovikov2017}, the application of our results to real analytic sets requires definability of their complex continuations.
For example, they may be applied to holomorphic-Pfaffian functions, that is functions admitting holomorphic continuations whose graph is Pfaffian as a real set (see also~\cite{Kaiser2016} for results regarding the definability of holomorphic extensions).
In this case, the complexity of the resulting parameterizations depends on the complexity of the graph of such a holomorphic continuation, rather than that of the original real sets.

We defer the discussion of the applications of \Cref{thm: intro cpt,thm: intro cprt} to real sets to~\cite{CarmonAnalyticallyGenerated}, where we define the appropriate subclass of sharply o-minimal structures, called \emph{analytically generated structures} (similar to those described in~\cite[Section 3.3]{BinyaminiNovikov2023}), and obtain sharp versions of the Yomdin--Gromov lemma and of the subanalytic preparation theorems of Parusinski and Lion--Rolin~\cite{Parusinski1994,LionRolin1997} (see also~\cite{vdDriesSpeisseger2002}), as well as establishing Wilkie's conjecture on polylogarithmic point counting for such structures.

In a subsequent paper~\cite{BinyaminiCarmonNovikovLE}, we use the preparation theorems of~\cite{CarmonAnalyticallyGenerated} to obtain sharp versions of the logarithmic--exponential preparation theorems of~\cite{LionRolin1997,vdDriesSpeisseger2002}.
We deduce from these the sharp o-minimality of certain structures, including~$\Rexp$, as well as establish Wilkie's conjecture for such structures.
In the case of $\Rexp$, this conjecture was recently resolved in~\cite{BinyaminiNovikovZak2024} by Binyamini, Novikov, and Zak.

This paper is organized as follows. 
In \Cref{sec:so minimal,sec: complex cells}, we review the necessary definitions and facts we need regarding sharp o-minimality and complex cells, as well as give the precise versions of \Cref{thm: intro cpt,thm: intro cprt}.
In \Cref{sec: proof outline,sec:fundamental lemma}, we outline the proofs of the main theorems and establish an analog of the \emph{fundamental lemma}~\cite[Lemma 22]{BinyaminiNovikov2019} that we use in the proofs.
Then, in \Cref{sec: main theorems proofs}, we give the proofs of the main theorems.
Finally, in \Cref{sec: real cells}, we recall the notions of real complex cells and real cellular maps and discuss the corresponding variants of the main theorems.

\begin{remark}[Asymptotic notation]
	We follow the convention in~\cite{BinyaminiNovikov2019,BinyaminiNovikovZak2022}, according to which each appearance of an expression of the form $c=\poly_{a}(b)$ means that $c$ is bounded from above by the value at $x=b$ of some polynomial $p_a(x)\in\rReal_{> 0}[x]$ depending only on $a$.
	This polynomial may be different at each occurrence of this notation, which only serves as shorthand.
	
	Similarly, an expression of the form $c=\order[a]{b}$ (respectively, $c=\order*[a]{b}$) means $c\leq C_a\cdot b$ (respectively, $c\geq C_a\cdot b$), where $C_a$ is some positive constant depending only on $a$, which may be different at each occurrence of this notation.
	Omitting $a$ from this notation (e.g.\ writing $c=\order{b}$) means that the relevant constant may be chosen independently of any other parameter.

	In particular, statements of the form ``If $c=\order{1}$, then the following holds'' should be read as ``There exists a constant $C$ such that if $c\leq C$, then the following holds'' and not ``If there exists a constant $C$ such that $c\leq C$, then the following holds''.

	When working with a fixed FD-filtration $\qty{\FD[\format][\degree]}$ (see \Cref{sec:so minimal}), we will suppress the dependence on the filtration from this notation.
  \end{remark}

\subsection{Acknowledgments}
The second author thanks Susheel Shankar for helpful discussions.

\addtocontents{toc}{\protect\setcounter{tocdepth}{2}}
\section{Preliminaries}
\label{sec: generalities}

\subsection{Sharply o-minimal structures}\label{sec:so minimal}
We use the framework of sharply o-minimal (\so-minimal, for short) structures, introduced by Binyamini, Novikov, and Zak in~\cite{BinyaminiNovikovZak2022}.
Examples of such structures include any reduct of the structure $\RrPfaff$, generated by restricted Pfaffian functions (see~\cite[Section 1.4.3]{BinyaminiNovikovZak2022}).
We assume familiarity with the standard notions of an o-minimal structure (by which we will always mean an o-minimal expansion of the real field) and o-minimal cell decomposition (see \cite{vdDries1998,Coste1999}).

In this section, we review the main definitions and facts related to \so-minimality, following~\cite{BinyaminiNovikovZak2022} up to some slight changes in terminology and presentation.

We associate pairs of numbers, called format and degree, to sets definable in an o-minimal structure.
These numbers encode the geometric complexity of the sets and serve as analogs for the ambient dimension and algebraic degree of an algebraic variety $\qty{P=0}\subset\rReal^n$, where $P\in\rReal[x_1,\dots ,x_n]$.
We record these as a double filtration on the collection of definable sets, as follows\footnote{We allow in the definition collections $\structure$ which are not structures, however this is a technicality meant to allow ``filtrations generated by a collection of sets'' which we define later in this section.}.

\begin{definition}
	\label{def: FD filtration}
	Let $\structure=\bigcup_{n\in\nNatural} \structure_n$, where $\structure_n$ is a collection of subsets of $\rReal^n$.
	A \emph{format--degree filtration} on $\structure$ (or \emph{FD-filtration} for short) is a collection $\{\structure_{\format,\degree}\}_{\format,\degree\in\nNatural}$ of sets satisfying the following conditions.
	\begin{enumerate}
		\item[(FD1)] Each $\structure_{\format,\degree}$ is a collection of sets in $\structure$.
		\item[(FD2)] Every set in $\structure$ belongs to some $\structure_{\format,\degree}$.
		\item[(FD3)] $\structure_{\format,\degree}\subset \structure_{\format+1,\degree} \cap \structure_{\format,\degree+1}$ for all $\format,\degree$,
		\item[(FD4)] If $A\subset \rReal^n$ is in $\structure_{\format,\degree}$, then $\format\geq n$.
	\end{enumerate}
	A set in $\structure$ is said to have \emph{format} $\format$ and \emph{degree} $\degree$ (with respect to a given FD-filtration) if it belongs to $\structure_{\format,\degree}$.
	Similarly, a function whose graph is in $\structure$ is said to have format $\format$ and degree $\degree$ if its graph belongs to $\structure_{\format,\degree}$.
\end{definition}
For the sake of brevity, we will also write $\qty{\FD[\format][\degree]}$ for the pair $(\structure,\qty{\FD[\format][\degree]}_{\format,\degree\in\nNatural})$ consisting of $\structure$ and an FD-filtration on it.

We will usually consider \Cref{def: FD filtration} when the collection $\structure$ is a structure.
The following is a quantitative version of the standard finiteness axiom for o-minimal structures, with polynomial dependence on the degree of a definable set.
\begin{definition}
	The FD-filtration $\qty{\FD[\format][\degree]}$ is \emph{sharply o-minimal} if there exists a collection of polynomials $\qty{P_{\format} \in\rReal[t]}_{\format\in\nNatural}$ such that the following condition holds.
	\begin{enumerate}\setcounter{enumi}{5}
		\item[(O)] If $A\subset \rReal$ and $A\in\structure_{\format,\degree}$, then $A$ has at most $P_{\format}(\degree)$ connected components.
	\end{enumerate}
\end{definition}

In particular, a structure $\structure$ admitting a sharply o-minimal FD-filtration is an o-minimal structure.
In~\cite{BinyaminiNovikovZak2022}, several notions of compatibility between an FD-filtration on a structure and the first-order construction of definable sets are introduced.
In this paper we will only need the following version.
We write $\complement{A}$ for the complement of a set $A\subset\rReal^n$ and $\pi(A)$ for its natural projection to~$\rReal^{n-1}$.

\begin{definition}
	Let $\structure$ be a structure and let $\{\FD[\format][\degree]\}$ be an FD-filtration on~$\structure$.
	We will say $\{\FD[\format][\degree]\}$ is \emph{sharply compatible} with $\structure$ if it satisfies the following conditions.
	For every $A\in\structure_{\format,\degree}$, we have
	\begin{enumerate}
		\item[(\s1)] $A\times\rReal, \rReal\times A\in \structure_{\format+1,\degree}$,
		\item[(\s2)] $\complement{A},\pi(A)\in\structure_{\format,\degree}$.
	\end{enumerate}
	If $A_1,\dots,A_k\subset \rReal^n$ are such that $A_i\in\structure_{\format,\degree_i}$, then
	\begin{enumerate}\setcounter{enumi}{4}
		\item[(\s3)] $\bigcap_i A_i,\,\bigcup_i A_i\in\structure_{\format,\,\,\sum D_i}$.
	\end{enumerate}
	If $P\in\rReal[x_1,\dots,x_n]$, then 
	\begin{enumerate}\setcounter{enumi}{3}
		\item[(\s4)] $\qty{P=0}\in\structure_{n,\deg P}$.
	\end{enumerate}
\end{definition}

More generally, given any $\structure=\bigcup_{n\in\nNatural} \structure_n$, where $\structure_n$ is a collection of subsets of $\rReal^n$, and an FD-filtration $\qty{\FD[\format][\degree]}$ on $\structure$, we may consider the structure generated by $\structure$ with the FD-filtration \emph{sharply generated} by $\qty{\FD[\format][\degree]}$, i.e.\ the FD-filtration obtained by closing $\qty{\FD[\format][\degree]}$ under the axioms (\s1)--(\s4).

\begin{definition}
	A structure $\structure$ equipped with a sharply compatible \so-minimal FD-filtration $\{\structure_{\format,\degree}\}_{\format,\degree\in\nNatural}$ is said to be \emph{sharply o-minimal}.
\end{definition} 

As for general o-minimal structures, the condition (O) on definable subsets of $\rReal$ implies the following similar bound for all definable sets.

\begin{proposition}[{\cite[Proposition 2.1]{BinyaminiNovikovZak2022}}]
	\label{prop: so minimality connected component bound}
	Let $\qty{\FD[\format][\degree]}$ be a sharply o-minimal structure and let $X\in\FD[\format][\degree]$.
	Then $X$ has at most $\polyfd$ connected components.
\end{proposition}

The following proposition, which follows directly from (FD1)--(FD4) and (\s1)--(\s4), may be used to bound the format and degree of sets and functions which are definable in a \so-minimal structure.

\begin{proposition}[{\cite[Proposition 1.13]{BinyaminiNovikovZak2022}}]\label{prop: so minimality formula complexity}
	Let $\qty{\FD[\format][\degree]}$ be a sharply generated FD-filtration on the structure $\structure$.
	Let $\psi$ be a first order formula containing $n$ different variables ($m$ of which are free) and whose atomic predicates are all of the form $(x_i\in X_i)$ for $X_i\in\FD[\format_i][\degree_i]$.
	Denote $\format=\max\qty{n,\format_1,\format_2,\dots}$ and $\degree=\sum \degree_i$.
	Then the subset of $\rReal^m$ defined by $\psi$ is in $\FD*[\format][\degree]$.
\end{proposition}

\begin{example}
	Let $\qty{\FD[\format][\degree]}$ be a sharply generated FD-filtration, let ${f:A\to B}$ and ${g:B\to C}$ be functions in $\FD[\format][\degree]$ and let ${B'\subset B}$ be in $\FD[\format][\degree]$.
	Then $A$, $f(A)$, $g\comp f$ and $f^{-1}(B')$ are all in $\FD*[\format][\degree]$.
\end{example}

The following notion of reduction between FD-filtrations is introduced in~\cite{BinyaminiNovikovZak2022} (for brevity, we suppress from our asymptotic notation the dependence on the relevant structures and FD-filtrations).
\begin{definition}\label{def: reduction}
	Let $\{\FD[\format][\degree][(1)]\},\{\FD[\format][\degree][(2)]\}$ be two FD-filtrations on structures $\structure^{(1)}$ and $\structure^{(2)}$, respectively.
	We say that $\{\FD[\format][\degree][(1)]\}$ is \emph{reducible} to $\{\FD[\format][\degree][(2)]\}$ if we have the inclusion $\FD[\format][\degree][(1)]\subset \FD*[\format][\degree][(2)]$ for all $F,D\in\nNatural$.

	If $\{\FD[\format][\degree][(1)]\}$ and $\{\FD[\format][\degree][(2)]\}$ are reducible to each other, we say they are \emph{equivalent}.
	In particular, the corresponding structures $\structure^{(1)}$ and $\structure^{(2)}$ are equal as collections of sets.
\end{definition}
\begin{remark}\label{rem: reductions}
The idea motivating \Cref{def: reduction} is that any bound of the form $\polyfd$ on some quantity with respect to the format and degree of a set in the filtration $\{\FD[\format][\degree][(2)]\}$ implies a bound of the same form with respect to the format and degree of the same set in the filtration $\{\FD[\format][\degree][(1)]\}$.
Similar implications hold for the other types of dependence on $\format$ and $\degree$ which appear in the statements of the theorems in this paper.

We will also use this idea in the following way.
During the proofs of the main theorems we will often replace sets and functions under discussion with ones satisfying additional assumptions, at the cost of increasing their complexity.
This may be done without loss of generality since, using \Cref{prop: so minimality formula complexity}, one may check that in all these cases we replace sets in $\FD[\format][\degree]$ by ones in $\FD*[\format][\degree]$.
\end{remark}

By~\cite[Theorem 1.9]{BinyaminiNovikovZak2022}, one may reduce every \so-minimal structure to a \so-minimal structure satisfying the following additional property.
\begin{definition}
	\label{def: sharp CD}
	A structure $\qty{\FD[\format][\degree]}$ has \emph{sharp cell decomposition} (or \s CD, for short) if for every collection $S_1,\dots,S_k\subset\rReal^n$ of sets in $\FD[\format][\degree]$ there is a decomposition of $\rReal^n$ into $\polyfd[\format][\degree, k]$ cells (in the sense of o-minimality), each of which is in $\FD*[\format][\degree]$ and compatible with each $S_i$.
\end{definition}

\begin{remark}\label{rem: complexity of component}
	It follows from~\cite[Proposition 1.23]{BinyaminiNovikovZak2022} that a sharply generated FD-filtration $\qty{\FD[\format][\degree]}$ has \s CD if and only if it is \so-minimal and has the property that each connected component of a set in $\FD[\format][\degree]$ is in $\FD*[\format][\degree]$.	
\end{remark}

\subsection{Complex cells}
\label{sec: complex cells}
We now give the basic definitions and notions relating to complex cells, mostly following~\cite{BinyaminiNovikov2019}, and then state our main result (\Cref{thm:cpt}).
As in~\cite[Section 2.2.2]{BinyaminiNovikov2019}, one also has a notion of \emph{real complex cells}, determined by holomorphic functions which are real valued on the real parts of their domains.
Throughout the paper, we will focus on the general complex setting and postpone to \Cref{sec: real cells} the discussion of the real setting.
Application of the real versions of the main theorems to real analytic sets as in~\cite[Corollaries 34 and 35]{BinyaminiNovikov2019} requires some additional work which is carried out in~\cite{CarmonAnalyticallyGenerated}.
For differences from the definitions of~\cite{BinyaminiNovikov2019}, see \Cref{rem: complex cell changes}.

We fix from now on some \s o-minimal structure $\qty{\structure_{\format,\degree}}$ admitting \s CD (see \Cref{sec:so minimal} for these notions).

\subsubsection{Cells and cellular maps}
Complex cells are certain subsets of $\cComplex^\ell$, defined by induction on $\ell$.
We start with the case $\ell=1$.

\begin{definition}\label{def: 1 dim complex cell}
	A \emph{complex cell} $\cell\subset\cComplex$ is one of the following sets, defined with respect to \emph{radii} $r,r_1,r_2\in\cComplex\setminus\{0\}$.
	
	\vspace*{\topsep}
	\newcolumntype{M}{r>{\!\!\!\!\!}}
	\begin{tabular}{@{$\quad\ \bullet\ \,$}l M l}
		A point:& $\point$&${}=\qty{0}$,\\
		A disc:& $\disc[r]$&${}=\qty{z\in\cComplex\st\abs{z}<\abs{r}}$,\\
		A punctured disc:& $\puncDisc[r]$&${}=\qty{z\in\cComplex\st 0<\abs{z}<\abs{r}}$,\\
		A disc complement:& $\puncDisc*[r]$&${}=\qty{z\in\cComplex\st 0< \abs{r}<\abs{z}}$,\\
		An annulus:& $\annulus[r_1][r_2]$&${}=\qty{z\in\cComplex\st 0<\abs{r_1}<\abs{z}<\abs{r_2}}$.
	\end{tabular}
	\vspace*{\topsep}\\
	The \emph{type} of $\cell$ is the symbol $\point$, $\disc$, $\puncDisc$, $\puncDisc*$ or $\annulus$, respectively.
\end{definition}
When using this notation, we will always implicitly assume that $r,r_1,r_2$ are non-zero and, in the case of annulus, that $\abs{r_1}<\abs{r_2}$.
Thus these sets are non-empty.

In order to define complex cells in $\cComplex^\ell$ for $\ell>1$, we use the following notation.
\begin{definition}[$\odot$-product]
	Let $X$ be a set and let $Y$ be a set-valued map defined on $X$.
	Then
	\begin{equation}
		X\odot Y = \qty{(x,y)\st x\in X,\  y\in Y(x)}.
	\end{equation}
\end{definition}
If the map $Y$ in the definition above is constant, then $X\odot Y = X\times Y(x_0)$ for any $x_0\in X$.

We extend the notation of \Cref{def: 1 dim complex cell} to the following set-valued functions.
\begin{definition}
	\label{def: fiber as map}
	Let $U\subset\cComplex^\ell$ and $r,r_1,r_2:U\to\cComplex\setminus\{0\}$.
	We write $\disc(r)$ for the map $U\to 2^\cComplex$ given by
	\begin{equation}
		\label{eq: moving fiber example}
		\disc(r)(x)=\qty{z\in\cComplex\st \abs{z}<\abs{r(x)}}\qc x\in U.
	\end{equation}
	We define $\puncDisc[r]$, $\puncDisc*[r]$ and $\annulus[r_1][r_2]$ analogously.
	We consider $\point$ as the constant map of value $\{0\}$.
	The \emph{type} of such a map is the same as that of its values.
\end{definition}
When using this notation, we will always implicitly assume that the maps $r,r_1,r_2$ are nowhere-vanishing on $U$ and, in the case of an annulus, that $\abs{r_1(x)}<\abs{r_2(x)}$ for all $x\in U$.
Thus the set given in \eqref{eq: moving fiber example} is non-empty for all $x\in U$.

We now define a complex cell in $\cComplex^\ell$, along with its type and length, in the following inductive manner.
\begin{definition}[Complex cell]
	A complex cell $\cell\subset\cComplex^0$ is equal to $\cComplex^0$, its type is the empty word and its length is $0$.

	Given a complex cell $\cell\subset\cComplex^\ell$, holomorphic functions $r,r_1,r_2:\cell\to\cComplex$ and $\cell[F]\in\qty{\point,\disc[r],\puncDisc[r],\puncDisc*[r],\annulus[r_1][r_2]}$, the set $\cell\odot\cell[F]$ is a complex cell in $\cComplex^{\ell+1}$.
	
	We call $\cell$ the \emph{base} of $\cell\odot\cell[F]$ and $\cell[F]$ its \emph{fiber}.
	The functions $r$ or $r_1,r_2$ are the \emph{radii} of $\cell[F]$.
	The type of $\cell\odot\cell[F]$ is given by the concatenation of the type of $\cell$ with the type of $\cell[F]$.
	The \emph{length} of a complex cell $\cell\subset\cComplex^\ell$ is $\ell$.
	The \emph{dimension} $\dim\cell$ of a complex cell $\cell$ of positive length is its dimension as a smooth complex manifold, which is equal to the difference between its length and the number of appearances of $\point$ in its type.
\end{definition}
When the type of a complex cell $\cell$ contains only discs, we will call $\cell$ a \emph{cellular polydisc}.

\begin{definition}
	\label{def: odot id}
	Let $\widehat\cell$ and $\cell\odot\cell[F]$ be complex cells of lengths $\ell$ and $\ell+1$, respectively, and let $f:\widehat\cell\to\cell$ be holomorphic.
	Then the pullback $\widehat\cell\odot\pullback{f}\cell[F]$ is the complex cell given by $\qty{(\varZ,w)\in\cComplex^{\ell+1}\st\varZ\in\widehat\cell,\ w\in\cell[F](f(\varZ))}$.
	In this case, we denote by $f\odot\idmap:\widehat\cell\odot\pullback{f}\cell[F]\to \cell\odot\cell[F]$  the map given by $(f\odot\idmap)(\varZ,w)=(f(\varZ),w)$.
\end{definition}

We will frequently use the following notation.
\begin{definition}
	\label{def: initial and fiber notation}
	Let $\varZ=(z_1,\dots,z_\ell)\in\cComplex^\ell$ and let $1\leq i \leq  j\leq \ell$.
	Then $\initial{\varZ}{i}{j}=(z_i,\dots,z_j)$.
	If $S\subset\cComplex^\ell$ then $\initial{S}{1}{k}=\qty{\initial{\varZ}{1}{k}\st \varZ\in S}$ is the projection of $S$ to the first $k$ coordinates.
	For any $\varW\in\initial{S}{1}{k}$, we denote the fiber of $S$ over $\varW$ by $S_\varW=\qty{\varZ\in S\st \initial{\varZ}{1}{k}=\varW}$. 
	If $S=\cell$ is a complex cell, we may identify $\cell_{\varW}$ with the complex cell of length $\ell-k$ obtained by omitting its first $k$ coordinates.
\end{definition}

We now define the notion of a cellular map between complex cells.
\begin{definition}[Cellular map]
	\label{def: cellular map}
	Let $\cell$ be a complex cell of length $\ell$.
	Let $f:\cell\to\cComplex^\ell$ be holomorphic with components $f=(f_1,\dots,f_\ell)$.
	The map $f$ is called \emph{cellular} if, for all $j=1,\dots,\ell$, we have that
	\begin{enumerate}
		\item The component $f_j$ depends only on the first $j$ coordinates, i.e.\ $f_j(\initial{\varZ}{1}{\ell})=f_j(\initial{\varZ}{1}{j})$ for all $\varZ\in\cell$;
		\item If the $j$-th coordinate of $\cell$ is not of type $\point$, then the derivative $\pdv{z_j}f_j$ does not vanish anywhere on $\cell$.
	\end{enumerate}
\end{definition}
The following consequences of this definition are straightforward to verify:
The composition of cellular maps is a cellular map;
Cellular maps preserve dimension;
If $f:\cell\to\cComplex^\ell$ is a cellular map, $\varZ\in\cell$ and $1\leq k\leq \ell$, then the restriction to the fibers $f:\cell_{\initial{\varZ}{1}{k}}\to\cComplex^\ell_{\initial{f(\varZ)}{1}{k}}$ is a cellular map.
\begin{definition}
	\label{def: prepared map}
	A cellular map $f:\cell\to\cComplex^\ell$ is \emph{prepared} in the $j$-th variable if $f_j$ is of the from
	\begin{equation}
		f_j(\varZ)=\pm z_j^{q_j}+g_j(\initial{\varZ}{1}{j-1})\qc \varZ\in\cell,
	\end{equation}
	where $q_j\in\zIntegers_{\neq 0}$ and $g_j:\initial{\cell}{1}{j-1}\to\cComplex$ is holomorphic.
	We call $g_j$ the \emph{center} of the prepared map in the $j$-th variable.
	When $q_j= 1$, we say that $f$ is a \emph{translate} in the $j$-th variable.
	When $g_j=0$, we say that $f$ is a \emph{power map} in the $j$-th variable.
	If $f$ is prepared (resp.\ a translate, a power map) in all variables, then we simply say that it is prepared (resp.\ a translate, a power map).
\end{definition}

\begin{remark}
	\label{rem: composing prepared maps}
	We note that the composition of two prepared maps is not necessarily prepared, but the composition of two cellular maps which are translates (resp.\ power maps) in the $j$-th variable is a translate (resp.\ power map) in the $j$-th variable.

	Likewise, precomposing a map which is prepared in the $j$-th variable with a power map in the $j$-th variable or precomposing a translate in the $j$-th variable with a map which is prepared in the $j$-th variable yields a map which is prepared in the $j$-th variable.
\end{remark}

\begin{remark}[Complexity of a complex cell]
	\label{rem: cell complexity}
	Throughout this paper, we will implicitly use the following convention when referring to the complexity of complex cells and of cellular maps (cf.~\cite[Section 2.2.3]{BinyaminiNovikov2019}).
	
	Let $\cell,\widehat\cell$ be complex cells of length $\ell$ and let $f:\cell\to\widehat\cell$ be a cellular map.
	All statements of the form ``$\cell\in\FD[\format][\degree]$'' should be interpreted as shorthand for ``The cell $\cell=\cell[F]_1\odot\cdots\odot\cell[F]_\ell$ as a subset of $\rReal^{2\ell}$, as well as the graphs of the radii of all fibers $\cell[F]_j$, are in $\FD[\format][\degree]$''.
	
	Similarly (and following the above convention with respect to $\cell,\widehat\cell$), all statements of the form ``$f\in\FD[\format][\degree]$'', for $f$ as above, should be interpreted as shorthand for ``The graph of $f$, as well as the complex cells $\cell,\widehat\cell$, are in $\FD[\format][\degree]$''.

	The reason for this convention is that it is not clear how to definably recover the graphs of the relevant radii functions (that is, their arguments and not just their absolute values) from the cell $\cell$ considered as a set.
	In~\cite{BinyaminiNovikov2019} and in this paper, complex cells are always constructed by constructing their radii functions.
	Thus one should consider these radii functions as part of the data constituting a complex cell.
\end{remark}

\subsubsection{Extensions and cellular covers}
A complex cell is always considered as a subset of some larger complex cell, called its extension, and defined inductively as follows.
\begin{definition}
	Let $0<\delta<1$.
	Any complex cell $\cell\subset\cComplex$ admits a \emph{$\delta$-extension} $\ext{\cell}{\delta}$, which is given by
	\begin{align}
		\begin{split}
		\ext{\point}{\delta}&=\point,\\
		\ext{\disc[r]}{\delta}&=\disc[r/\delta],\\
		\ext{\puncDisc[r]}{\delta}&=\puncDisc[r/\delta],\\
		\ext{\puncDisc*[r]}{\delta}&=\puncDisc*[\delta r],\\
		\ext{\annulus[r_1][r_2]}{\delta}&=\annulus[\delta r_1][r_2/\delta].
		\end{split}
	\end{align}
	We call $\delta$ the \emph{extension parameter} of the extension $\ext{\cell}{\delta}$.
\end{definition}
\begin{definition}\label{def: delta extension}
	Let $(\delta_1,\dots,\delta_{\ell+1})\in(0,1)^{\ell+1}$.
	A complex cell $\cell\odot\cell[F]$ of length $\ell+1$ is said to admit a $(\delta_1,\dots,\delta_{\ell+1})$-extension if both of the following conditions are met:
	\begin{enumerate}
		\item The base $\cell$ admits a $(\delta_1,\dots,\delta_{\ell})$-extension;
		\item The radii defining the fiber $\cell[F]$ analytically continue to holomorphic functions on $\ext{\cell}{(\delta_1,\dots,\delta_{\ell})}$ such that $\cell[F]$ remains well-defined.
	\end{enumerate}
	By this last condition we mean that the radii defining $\cell[F]$ are non-zero for all $\varZ\in\ext{\cell}{(\delta_1,\dots,\delta_{\ell})}$ and, in the case of an annulus, that the outer radius has larger absolute value than the inner radius for all $\varZ\in\ext{\cell}{(\delta_1,\dots,\delta_{\ell})}$.
	In this case, we set
	\begin{equation}
		\ext*{\cell\odot\cell[F]}{(\delta_1,\dots,\delta_{\ell+1})}=\ext{\cell}{(\delta_1,\dots,\delta_{\ell})}\odot\ext{\cell[F]}{\delta_{\ell+1}}.
	\end{equation} 

	Finally, for $0<\delta<1$, we abbreviate $\ext*{\cell\odot\cell[F]}{(\delta,\dots,\delta)}$ as $\ext*{\cell\odot\cell[F]}{\delta}$.
\end{definition}

We will also use the same extension notation for Euclidean balls, as follows.
\begin{definition}
	Let $\ball[\varZ][r]=\qty{\varW\in\cComplex^\ell\st\abs{\varW-\varZ}<r}$ be a Euclidean ball.
	We define its $\delta$-extension to be $\ext{\ball[\varZ][r]}{\delta}=\ball[\varZ][r/\delta]$. 
\end{definition}
For balls and cellular polydiscs, we will use this extension notation also in the case where $\delta\geq 1$.

\begin{definition}
	\label{def: weierstrass cell}
	Let $0<\gamma<1$, let $\cell\odot\ext{\cell[F]}{\gamma}$ be a complex cell where $\cell[F]\neq \point$ and let $Z\subset\cell\odot\ext{\cell[F]}{\gamma}$. 
	We say that $\cell\odot\ext{\cell[F]}{\gamma}$ is a \emph{Weierstrass cell} for $Z$ with \emph{gap} $\gamma$ if we have
	\begin{equation}
		Z\cap (\cell\odot(\ext{\cell[F]}{\gamma}\setminus\cell[F]))=\emptyset.	
	\end{equation}
\end{definition}
\begin{remark}
	If $\cell\odot\ext{\cell[F]}{\gamma}$ is a Weierstrass cell for $Z\subset\cell\odot\ext{\cell[F]}{\gamma}$ and $\cell[F]=\disc,\annulus$, then $\closure{Z}$ does not intersect $\cell\odot\boundary(\ext{\cell[F]}{\gamma})$.
	However, for $\cell[F]=\puncDisc$, we may have that $\closure{Z}$ intersects $\cell\odot\point$.
\end{remark}

The hyperbolic geometry of a complex cell inside its extension plays a central role in the constructions of~\cite{BinyaminiNovikov2019} and of this paper.
Our main results are therefore formulated in terms of the following alternative definition of an extension, which is better adapted to the hyperbolic viewpoint.
\begin{definition}
	\label{def: rho extension}
	Let $\rho >0$.
	The $\hyperbolicParameter{\rho}$-extension $\hExt{\cell}{\rho}$ of a complex cell $\cell\subset\cComplex$ is the extension $\ext{\cell}{\delta}$, where $\delta$ and $\rho$ are related by
	\begin{alignat}{3}
		\quad\quad\ \,\rho &= \frac{2\pi\delta}{1-\delta^2},&\quad\quad \delta&=\frac{\sqrt{\pi^2+\rho^2}-\pi}{\rho}&&\quad\quad \text{if $\cell=\disc$},\\
		\rho &= \frac{\pi^2}{2\abs{\log\delta}},&\quad\quad \delta&=e^{-\pi^2/2\rho}&&\quad\quad\text{if $\cell=\puncDisc,\puncDisc*,\annulus$}.
	\end{alignat}
	We call $\rho$ the \emph{hyperbolic extension parameter} of the extension $\hExt{\cell}{\rho}$.
	We will also write $\hyperbolicParameter{\rho}$ for the corresponding (Euclidean) extension parameter $\delta$.
	Extensions $\hExt{\cell}{\rho_1,\dots,\rho_\ell}$ and $\hExt{\cell}{\rho}$ of a complex cell $\cell\subset\cComplex^\ell$ are defined as in \Cref{def: delta extension}.
\end{definition}
This normalization is chosen such that, in the one dimensional case, the boundary components of $\cell$ have length at most $\rho$ in the hyperbolic metric of $\hExt{\cell}{\rho}$ (see~\cite[Fact 6]{BinyaminiNovikov2019}).

\begin{remark}
	We note that, for both $\delta$-extensions and $\hyperbolicParameter{\rho}$-extensions, a smaller extension parameter corresponds to a larger extension.
\end{remark}

When using the notation of \Cref{def: delta extension,def: rho extension}, we will always implicitly assume that reference to a complex cell $\ext{\cell}{\delta}$ implies that  $\cell$ is a complex cell admitting a $\delta$-extension, and similarly for $\hyperbolicParameter{\rho}$-extensions.

\begin{definition}[Cellular cover]
	A finite collection of cellular maps $\{\phi_j:\ext{\cell_j}{\delta'}\to\ext{\cell}{\delta}\}$ is called a \emph{cellular cover} if $\cell\subset\bigcup_j \phi_j(\cell_j)$.
\end{definition}
Cellular covers $\{\phi_j:\ext{\cell_j}{\delta'}\to\ext{\cell}{\delta}\}$ and $\{\varphi_{jk}:\ext{\cell_{jk}}{\delta''}\to\ext{\cell_j}{\delta'}\}$ may be composed to yield a cellular cover $\{\phi_j\comp\varphi_{jk}:\ext{\cell_{jk}}{\delta''}\to\ext{\cell}{\delta}\}$.

Let $\ext{\cell}{\delta}$ be a complex cell and let $0<\delta'<1$.
If $\delta'>\delta$, then $\cell$ also admits a $\delta'$-extension, essentially by ``forgetting'' the larger extension corresponding to $\delta$.
If $\delta'<\delta$, then one may use the following sharp version of the refinement theorem \cite[Theorem 9]{BinyaminiNovikov2019} to move to a cell of extension $\delta'$.

\begin{theorem}[Sharp refinement theorem]\label{thm:sharp refinement}
	Let $0 < \sigma < \rho$ and let $\hExt{\cell}{\rho}$ be a complex cell of length $\ell$ in $\FD[\format][\degree]$.
	Then there exists a cellular cover $\{f_j:\hExt{\cell_j}{\sigma}\to\hExt{\cell}{\rho}\}$ of size $\polyfd[\ell][\rho, 1/\sigma]$ such that each $f_j$ is a translate map and is in~$\FD*[\format*][\degree]$.
\end{theorem}
\begin{proof}
	The proof of the refinement theorem in~\cite[Section 6.1]{BinyaminiNovikov2019} extends directly to the \so-minimal setting.

	Fibers of type $\puncDisc*$ are treated similarly to annuli of constant radii, rather than by inverting and reducing to the $\puncDisc$ case.
	This is needed in order to obtain cellular maps which are translates.
\end{proof}

\begin{remark}
	\label{rem: extension complexity}
	We note that if $\cell\in\FD[\format][\degree]$ is a complex cell admitting an extension $\ext{\cell}{\delta}$ for some $0<\delta<1$, one cannot, in general, bound the complexity of $\ext{\cell}{\delta}$ independently of $\delta$.
	For example, consider the cell $\cell=\disc[1]\odot\disc[\exp(\exp(z_1))]$ (in the structure $\RrPfaff$, say).
	This cell admits a $\delta$-extension for any value of $0<\delta<1$.
	However, the intersection of $\{z_2=e\}$ and $\ext{\cell}{\delta}$ is of size $\order*{\delta^{-1}}$. 

	In comparison, the cells constructed in \Cref{thm:sharp refinement} above are obtained by restricting given holomorphic functions to smaller cells admitting a larger extension, and do not correspond to analytically continuing the radii functions defining a given cell.
\end{remark}

\begin{remark}[Normalized cells and inversions]
	\label{rem: normalized cells}
	Let $\cell\in\FD[\format][\degree]$ be a complex cell of length $\ell$.
	We may normalize the coordinates of the cell by a biholomorphic map in the following way, for all $1\leq j\leq \ell$:
	\begin{itemize}
		\item If the $j$-th coordinate of $\cell$ is of type $\disc$, $\puncDisc$ or $\puncDisc*$, then its radius is the constant function $1$.
		\item If the $j$-th coordinate of $\cell$ is of type $\annulus$, then its outer radius is the constant function $1$.
	\end{itemize}
	In the case of annulus, we may also normalize its inner radius instead of the outer radius (note that we can not, in general, fix both of these radii to be constants since the logarithmic width $\abs{\log\abs{r_2}-\log\abs{r_1}}$ is a conformal invariant of the annulus $\annulus[r_1][r_2]$).

	In a similar manner, one may apply an inversion in a coordinate of type $\annulus$ in order to exchange between the inner and outer boundary components of the corresponding annulus.
	Inversion may also be used to move between coordinates of type $\puncDisc$ and $\puncDisc*$.

	These transformations are all in $\FD*[\format][\degree]$ and are equivariant with respect to taking extensions.
\end{remark}

\begin{remark}
	\label{rem: complex cell changes}
	Our exposition of complex cells has mostly followed~\cite{BinyaminiNovikov2019}, up to some small changes in terminology and up to the following additional differences.

	First, we introduce the fiber type $\puncDisc*$ to allow for unbounded cells.
	We note that fibers of this type appear implicitly in the clustering constructions of~\cite[Section 6.3]{BinyaminiNovikov2019}.
	Making these explicit will be useful in~\cite{CarmonAnalyticallyGenerated}, where we obtain preparation theorems as described in~\cite[Section 4.2]{BinyaminiNovikov2019} for functions with possibly unbounded domains.
	Working with these unbounded domains is needed for our adaptation of the methods of~\cite{vdDriesSpeisseger2002} in~\cite{BinyaminiCarmonNovikovLE}.

	Second, the definition of a cellular map (\Cref{def: cellular map}) is similar to the less restrictive version used in~\cite{Binyamini2024}.
	This is needed in order to resolve a minor technicality in the proof of the CPrT \cite[Theorem 7]{BinyaminiNovikov2019} (cf.\ \Cref{sec: CPrT}). 
	
	Lastly, the definition of a prepared map (\Cref{def: prepared map}) now allows for a sign (as in~\cite{Shankar2025}) and for a negative exponent.
	We note that definability assumptions on the holomorphic functions appearing in the statements throughout this paper ensure that one cannot encounter, e.g., an essential singularity around the origin of a punctured disc or at infinity for a disc complement.
\end{remark}

\begin{remark}
	\label{rem: complex basic fiber}
	It will occasionally be convenient to work with sets of the form $\cell\odot\cComplex^k\odot\cell[F]$, where $\cell\odot\cell[F]$ is a complex cell and $k\in\nNatural$.
	In all such cases, we will consider this notation as shorthand for the union of $2^k$ complex cells corresponding to the (non-disjoint) decomposition $\cComplex=\disc[2]\cup\puncDisc*[1]$, where the fiber $\cell[F]$ depends only on the coordinates of $\cell$.
	We note that if $\cell\odot\cell[F]$ admits a $\delta$-extension, then so do each of the cells corresponding to $\cell\odot\cComplex^k\odot\cell[F]$.
\end{remark}

\begin{remark}
	Let $\cell\in\FD[\format][\degree]$ be a complex cell of length $\ell$.
	Since $\format$ is at least $2\ell$, any bound of the form, say, $\polyfld$ may be weakened to a bound of the form $\polyfd$.
	We will thus occasionally suppress such dependence on $\ell$ from the notation when it is not needed.	
	We do, however, retain this explicit dependence on $\ell$ in the statements of some theorems since this is important for their proofs, which proceed by induction on $\ell$.
\end{remark}

\subsubsection{Sharp cellular parameterization}
Our main result in the complex setting, \Cref{thm:cpt} below, is a sharp version of the cellular parameterization theorem (CPT)~\cite[Theorem 8]{BinyaminiNovikov2019}.
We first require the following two definitions.

\begin{definition}
	A map $f:X\to Y$ is \emph{compatible} with a subset $Z\subset Y$ if $f(X)$ is either contained in $Z$ or disjoint from $Z$.
\end{definition}

\begin{definition}
	Let $U\subset\cComplex^\ell$ be a domain.
	An \emph{analytic subset} $Z\subset U$ is a relatively closed subset that is locally given by the common zero-set of finitely many holomorphic functions.

	If, in addition, $Z$ is of complex codimension $1$, then we say it is an \emph{analytic hypersurface}.
\end{definition}
By~\cite[Proposition 2.6, Theorem 2.8]{Chirka1989}, an analytic subset is a hypersurface if and only if it is locally given by the zero-set of a single holomorphic function which does not vanish identically.

\begin{theorem}[Sharp Cellular Parametrization Theorem, \s CPT]\label{thm:cpt}
	Let ${\rho,\sigma>0}$. 
	Let $\hExt{\cell}{\rho}$ be a cell of length $\ell$ and let $Z_1,\dots,Z_M\subset\hExt{\cell}{\rho}$ be analytic hypersurfaces such that $\hExt{\cell}{\rho}$ and $Z_1,\dots,Z_M$ are in $\FD[\format][\degree]$.
	Then there exists a cellular cover ${\{f_j:\hExt{\cell_j}{\sigma}\to\hExt{\cell}{\rho}\}}$ of size $\poly_{\format*}(M, \degree, \rho, 1/\sigma)$ such that each $f_j$ is prepared, compatible with each $Z_i$, and is in $\FD*[\format*][\degree]$.
\end{theorem}

\begin{remark}\label{rem : no M in complexity}
	We note that, unlike in the statement of the CPT in~\cite{BinyaminiNovikov2019}, the complexity of the maps $f_j$ in \Cref{thm:cpt} does not depend on $M$.
	This is crucial for our inductive arguments  --- as part of the proof, we inductively apply the \s CPT with respect to a cell of smaller length and a large number of hypersurfaces.
	This is similar to the model-theoretic notion of \emph{distality} (see~\cite[Section 3.4]{Starchenko2017})
\end{remark}

The proof of the \s CPT is given in \Cref{sec: main theorems proofs}.
An outline of the proof is given in \Cref{sec: proof outline}.
The main obstacle in extending the proof of~\cite[Theorem 8]{BinyaminiNovikov2019} to the \so-minimal setting is reducing to the case of a Weierstrass cell $\hExt{\cell}{\rho}\odot\ext{\cell[F]}{\gamma}$ for $\bigcup Z_i$ with gap $\gamma=1-1/\polyfd[\format*][\degree,M]$ (see \Cref{def: weierstrass cell}).

As in~\cite{BinyaminiNovikov2019}, it is enough to first consider a slightly weaker version of the \s CPT where we do not assume that the cellular maps in the resulting cellular covers are prepared.
The stronger version, where these cellular maps are prepared, follows from the following sharp analog of~\cite[Theorem 7]{BinyaminiNovikov2019}\footnote{The dependence on $\sigma$ in the statement of \Cref{thm:sharp preparation} should also be included in the original statement of~\cite[Theorem 7]{BinyaminiNovikov2019}.
}.

\begin{theorem}[Sharp Cellular Preparation Theorem, \s CPrT]\label{thm:sharp preparation}
	Let ${\rho,\sigma>0}$.
	Let ${f:\hExt{\cell}{\rho}\to\cComplex^\ell}$ be a cellular map in $\FD[\format][\degree]$.
	Then there exists a cellular cover $\{g_j:\hExt{\cell_j}{\sigma}\to\hExt{\cell}{\rho}\}$ of size $\poly_{\format*}\qty(\degree,\rho,1/\sigma)$ such that each $g_j$ is in $\FD*[\format*][\degree]$ and each $f\comp g_j$ is prepared.
\end{theorem}

The proof of~\cite[Theorem 7]{BinyaminiNovikov2019} essentially extends directly to the \so-minimal setting.
However, we will give a slightly more streamlined version of the proof in~\Cref{sec: CPrT} adapted to our presentation and definitions, which also resolves a small technical issue in the original proof.

\subsection{Outline of the proof of the \texorpdfstring{\s CPT}{\#CPT}}\label{sec: proof outline}
The proof of the \s CPT (\Cref{thm:cpt}) proceeds by induction on the length $\ell$ of the cell $\cell$.
In order to make the structure of the proof clearer, we divide the \s CPT into the following two families of statements, which we prove by simultaneous induction on $\ell$.
\begin{enumerate}[leftmargin=6\parindent,align=right]
	\item[$(\text{\s CPT}_1)_{\leq\ell}$] The \s CPT for a Weierstrass cell $\hExt{\cell}{\rho}\odot\ext{\cell[F]}{\gamma}$ for $\bigcup Z_i$ with gap $\gamma=1-1/\polyfd[\format*][\degree,M]$, where the length of $\cell$ is at most $\ell$.
	\item[$(\text{\s CPT}_2)_{\leq\ell}$] The \s CPT for cells $\cell\odot\cell[F]$ where the length of $\cell$ is at most $\ell$. 
\end{enumerate}
Similarly to the ``$\leq\ell$'' versions described above, we have the corresponding ``$<\ell$'' versions of each statement, defined in the obvious manner.

We divide the proof of the \s CPT into the following two inductive steps.
\begin{align}
	(\text{\s CPT}_2)_{<\ell}&\implies 
	(\text{\s CPT}_1)_{\leq\ell} \tag{Step $1$},\\
	(\text{\s CPT}_2)_{<\ell} + (\text{\s CPT}_1)_{\leq\ell}&\implies 
	(\text{\s CPT}_2)_{\leq\ell}\tag{Step $2$}.
\end{align}

The proof of Step 1 mostly follows the proof of the CPT in~\cite[Section 8]{BinyaminiNovikov2019}.
The main difference is that we do not treat the case of $M$ analytic sets $Z_1,\dots, Z_M$ by reduction to the case $M=1$, instead proving the general case directly.
This is needed in order to get covering cells of complexity independent of $M$ (see \Cref{rem : no M in complexity}).
Furthermore, some extra care is needed to ensure that certain discriminant sets are of low complexity, as well as when reducing to the case of proper projections for $\cell[F]=\puncDisc,\puncDisc*$.

The harder part of the proof is in Step 2.
In the rest of this section, we give an overview of the proof of this step and introduce several necessary ingredients.
The full details are given in \Cref{sec: main theorems proofs}.

Let $\hExt*{\cell\odot\cell[F]}{\rho}$ be a complex cell of length $\ell+1$ where $\cell$ is of length $\ell$.
Let $Z_1,\dots, Z_M\subset \hExt*{\cell\odot\cell[F]}{\rho}$ be analytic hypersurfaces, such that $\cell\odot\cell[F]$ and $Z_1\dots,Z_M$ are in $\FD[\format][\degree]$.
If the type of $\cell\odot\cell[F]$ contains $\point$, the claim follows immediately by induction on $\ell$ as we may consider the corresponding cell with all $\point$ fibers omitted, which is of smaller length.
Thus we may assume that the type of $\cell\odot\cell[F]$ does not contain $\point$.

By first applying the refinement theorem (\Cref{thm:sharp refinement}) and pulling back by the resulting cellular covers, we may reduce to the case where $\rho$ is equal to some constant (see \Cref{rem: reductions}).
This will only increase the size of the resulting cover by a factor of $\polyl(\rho)$ and will not qualitatively increase the complexity of the resulting cellular maps.

Similarly, by applying the refinement theorem at the end of the proof, we may assume throughout the proof that $\sigma$ is equal to some constant.

For the sake of exposition, we consider in this outline the case where $M=1$ and $\cell[F]=\annulus$.
We normalize the outer radius of $\cell[F]$ and assume that it is identically equal to $1$ (see \Cref{rem: normalized cells}).
For the rest of this section, we set $Z=Z_1$ and~$\hyperbolicParameter{\rho}=1/6$.
We denote by $\projfiber$ and $\projbase$ the projection maps given by $\projfiber(\initial{\varZ}{1}{\ell+1})=z_{\ell+1}$ and $\projbase(\initial{\varZ}{1}{\ell+1})=\initial{\varZ}{1}{\ell}$, for $\initial{\varZ}{1}{\ell+1}\in\cComplex^{\ell+1}$.

\subsubsection{\texorpdfstring{Removing an $\epsilon$-net}{Removing an ε-net}}\label{sec: removing an epsilon net intro}
In order to reduce to the case of a Weierstrass cell, where the set $Z$ cannot get too close to the boundary of the fiber $\cell[F]$, we will first reduce to a case where connected components of $Z$ do not move much in the radial direction of the last coordinate.

We will focus on the outer boundary of $\cell[F]$.
The inner boundary may be treated by inverting the last coordinate and repeating the construction.

The following lemma, similar to the \emph{fundamental lemmas} of~\cite[Lemmas 20--22]{BinyaminiNovikov2019}, will allow us to bound the diameter of $Z$ in the radial direction of the last coordinate, assuming its values in the last coordinate avoid an $\epsilon$-dense set near the outer boundary component of $\cell[F]$, for some $\epsilon>0$.
We denote by $\diam{X}[\cComplex]$ the diameter of a subset $X\subset\cComplex$ with respect to the standard Euclidean metric.

\begin{lemma}[{Fundamental Lemma for $\latticeComplement[\epsilon]$}]
	\label{lem:fundamental lemma}
	Let $\epsilon>0$, let $\hExt{\cell}{\rho}$ be a complex cell of length $\ell$ where $\rho=\order{1}$, and let $f:\hExt{\cell}{\rho}\to\cComplex\setminus\epsilon\zIntegers^2$ be holomorphic.
	Then
	\begin{equation}\label{eq:fundamental lemma}
		\diam{f(\cell)}[\cComplex]=\order[\ell]{\epsilon}.
	\end{equation}
\end{lemma}
The fundamental lemma is proved in \Cref{sec:fundamental lemma}.
We remark that the proof, which is similar to the proofs of the fundamental lemmas of~\cite{BinyaminiNovikov2019}, may be easily adapted to yield better estimates depending on the size of $\rho$.
Since \eqref{eq:fundamental lemma} suffices for our arguments, we do not prove these more refined estimates.

In order to apply this lemma, we first reduce to the case where the projection $\projfiber(Z)$ of $Z$ to the last coordinate does not intersect the set $\epsilon\zIntegers^2\cap\annulus[1][6]$.
It will be convenient to consider $Z$ as a subset of $\ext{\cell}{1/6}\odot\annulus[1][6]$, which we view as the $1/2$-extension of $\ext{\cell}{1/3}\odot\annulus[2][3]$.
Under this assumption, we will establish the following \emph{radial bound}:
Let $Z_0$ be a connected component of the intersection of $Z$ with $\ext{\cell}{1/3}\odot\annulus[2][3]$.
Then
\begin{equation}\label{eq: radial bound outline}
	\diam{\operatorname{abs}({\projfiber}(Z_{0}))}[\rReal]<\epsilon \cdot \polyfld,
\end{equation}
where we denote by $\diam{X}[\rReal]$ the diameter of the subset $X\subset\rReal$ with respect to the standard Euclidean metric, and where $\operatorname{abs}:\cComplex\to\rReal$ is the absolute value map. 

\begin{figure}
	\includegraphics[width=0.66\linewidth]{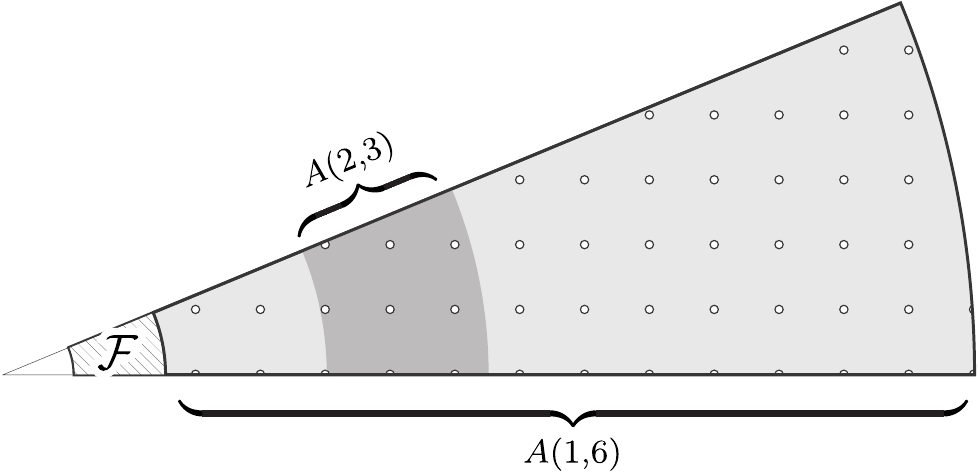}
	\caption{\emph{Removing an $\epsilon$-net.}
	By inductively applying the \s CPT in the base, we reduce to the case where the projection of $Z$ to the last coordinate does not intersect the set $\epsilon\zIntegers^2\cap\annulus[1][6]$.}
	\label{fig: epsilon net}
\end{figure}

A suitable choice of $\epsilon^{-1}=\poly_{\format*}(D)$ in \eqref{eq: radial bound outline}, together with estimates on the number of connected components $Z_0$ as above, yields a Weierstrass cell for $Z$ with gap $\gamma=1-1/\polyfd[\format*][\degree]$ (see \Cref{sec: RB}), allowing us to finish the proof by applying $\text{\s CPT}_1$.

The reduction to the case where $\projfiber(Z)$ does not intersect $\epsilon\zIntegers^2\cap\annulus[1][6]$ proceeds by first inductively applying $\text{\s CPT}_2$ in the base $\ext{\cell}{1/6}$ with respect to the projections $\projbase(Z\cap\Gamma_p)$, where $\Gamma_p$ is the graph of the constant function over $\ext{\cell}{1/6}$ of value $p\in\epsilon\zIntegers^2\cap\annulus[1][6]$.
We then pull back by a cellular map in the resulting cellular cover.

The image of this map is either contained in one of the projections $\projbase(Z\cap\Gamma_p)$, in which case we may proceed by induction on dimension, or else it is disjoint from all these projections.

\subsubsection{Radial bound}

Under the assumptions of the previous section, the radial bound \eqref{eq: radial bound outline} will follow from an application of the fundamental lemma (\Cref{lem:fundamental lemma}) to the composition of $\projfiber$ and a collection of holomorphic (not necessarily cellular) maps, defined on complex cells of a suitable extension and parameterizing the set $Z$.

In order to construct these parameterizing maps, we will first reduce to the case where the type of $\cell$ contains only $\disc,\annulus$ and then to the case where the type of $\cell$ contains only $\disc$.

This second reduction, carried out in \Cref{sec: reduce to only discs}, proceeds by a maximum-modulus type result which lets us control the radial diameter of $\projfiber(Z)$ by considering only those points of $Z$ which lie over the \emph{skeleton} of the base $\ext{\cell}{1/3}$ (i.e.\ the $\odot$-product of its boundary components, see \Cref{def: skeleton}).
As in the proof of the refinement theorem (\Cref{thm:sharp refinement}), we may cover this skeleton by $\order[\ell]{1}$ cellular polydiscs, whose $1/6$-extensions lie in $\ext{\cell}{1/6}$.

Having reduced to the case where the type of $\cell$ contains only $\disc$, we may assume (see \Cref{rem: normalized cells}) that $\cell=\disc[1]^{\times \ell}$.  
We now wish to construct the holomorphic maps parameterizing $Z$.
Since we do not need these maps to be cellular, we may allow ourselves to rotate $Z$ by a unitary transformation. 
The following proposition shows that one can (locally) rotate $Z$ such that it lies in a Weierstrass cell which is not too small compared to the complexity of $Z$.
In the statement of this proposition, a \emph{polydisc} refers simply to a cartesian product of discs, and not to the more general cellular polydiscs.

\begin{proposition}[{\cite[Lemma 20]{BinyaminiNovikov2017}}]\label{prop: Weierstrass polydiscs}
	Let $\ball\subset\cComplex^\ell$ be a Euclidean ball centered at the origin and let $Z\subset\ball$ be an analytic hypersurface such that $Z\in\FD[\format][\degree]$.
	Then there exists $\gamma=1-1/\polyfld$, a polydisc $\ext{\ball}{1/(1-\gamma)}\subset\polydisc\subset\ball$ and a unitary transformation $U:\cComplex^\ell\to\cComplex^\ell$ such that $\polydisc$ is a Weierstrass cell for $\polydisc\cap U^{-1}(Z)$ with gap $\gamma$.
\end{proposition}
\begin{proof}
	Denoting $\eta=1/(1-\gamma)$, the proof in~\cite{BinyaminiNovikov2017} extends directly to the \so-minimal setting, since the natural analog of~\cite[Corollary 17]{BinyaminiNovikov2017} holds by \so-minimality.
\end{proof}
We note that we may choose $\gamma$ in the statement of the proposition above to be at least $\order*{1}$.

The assumption that $\cell=\disc[1]^{\times \ell}$ allows us to cover ${\ext{\cell}{1/3}\odot\annulus[2][3]}$ by a collection $\mathcal{B}$ of $\polyfld$ balls whose $\delta$-extensions lie in $\ext{\cell}{1/6}\odot\annulus[1][6]$, for a suitable value of $\delta$ chosen using \Cref{prop: Weierstrass polydiscs}.
Fixing one of these balls, $\ball\in\mathcal{B}$, we assume without loss of generality that it is centered at the origin and apply \Cref{prop: Weierstrass polydiscs} to obtain a unitary transformation $U$ and a Weierstrass polydisc $\polydisc$ for $U^{-1}(Z)$ with an appropriate gap $\gamma=1-1/\polyfld$.

We now apply $\text{\s CPT}_1$ with respect to $\polydisc$ and the rotated set $U^{-1}(Z)$ to cover it by images of holomorphic maps $f_j$ defined on complex cells $\ext{\cell_j}{\delta'}$, each of them admitting an extension large enough in order to apply the fundamental lemma and such that the composition of $f_j$ with $\projfiber\comp U$ maps into $\latticeComplement[\epsilon]$.

\subsection{Proof of the fundamental lemma}\label{sec:fundamental lemma}
In this section we give the proof of the fundamental lemma for $\latticeComplement[\epsilon]$ (\Cref{lem:fundamental lemma}).
We begin with some notation needed for this section, as well as recalling several notions and results from \cite[Section 5]{BinyaminiNovikov2019}.
Let $f:X\to Y$ be a continuous map between path-connected spaces.
We write $\fundamentalGroup(X)$ for the fundamental group of $X$ and let $\pushforward{f}:\fundamentalGroup(X)\to\fundamentalGroup(Y)$ be the homomorphism induced by $f$.
For $r\in\cComplex\setminus\qty{0}$, we write $\circle[r]=\qty{z\in\cComplex\st\abs{z}=\abs{r}}$.

Let $X\subset\cComplex$ be a non-empty open subset such that $\cComplex\setminus X$ contains at least $2$ points.
The set $X$ admits a unique hyperbolic metric of constant curvature $-4$.
We denote by $\diam{Y}[X]$ the diameter of $Y\subset X$ with respect to this metric.
Given two such sets $X_1, X_2$, any holomorphic map $f:X_1\to X_2$ is (weakly) contracting with respect to the corresponding pair of hyperbolic metrics, by the Schwarz--Pick lemma.
In particular, if $f$ is a biholomorphism, then it is also an isometry.

\begin{definition}[Skeleton of a cell, {\cite[Definition 39]{BinyaminiNovikov2019}}]
	\label{def: skeleton}
	Let $\cell\odot\cell[F]$ be a complex cell  whose type contains only $\disc,\annulus$.
	Its \emph{skeleton}, denoted $\skeleton{\cell\odot\cell[F]}$, is given by $\skeleton{\cell}\odot\boundary\cell[F]$, where
\begin{equation}
	\boundary\disc[r]=\circle[r]\qc \boundary\annulus[r_1][r_2]=\circle[r_1]\cup\circle[r_2],
\end{equation}
and $\skeleton{\cComplex^0}=\cComplex^0$.
\end{definition}
When using the notation $\skeleton{\cell}$, we will always implicitly assume that the type of $\cell$ does not contain $\puncDisc$, $\puncDisc*$ or $\point$.

In the following two propositions, $\hExt{\cell}{\rho}$ is a complex cell and $f:\hExt{\cell}{\rho}\to X$ is a holomorphic map where $X$ is a hyperbolic Riemann surface (i.e.\ a Riemann surface admitting a hyperbolic metric, as above).
\begin{proposition}[{\cite[Lemma 41]{BinyaminiNovikov2019}}]\label{prop: skeleton controls boundary}
	We have the inclusion
	\begin{equation}
		\boundary f(\cell)\subset f(\skeleton{\cell}).
	\end{equation}
\end{proposition}
\begin{proposition}[{\cite[Lemma 42]{BinyaminiNovikov2019}}]\label{prop: trivial fundamental group}
	If $\pushforward{f}(\fundamentalGroup(\cell))\subset\fundamentalGroup(X)$ is trivial, then 
	\begin{equation}
		\diam{f(\cell)}[X]=\order[\ell]{\rho}.
	\end{equation}
\end{proposition}

We have the following lemma, analogous to~\cite[Lemma 44]{BinyaminiNovikov2019}.
\begin{lemma}\label{lem:shortest loop}
	Let $\gamma:[0,1]\to\latticeComplement$ be a loop of hyperbolic length $\order{1}$.
	If $\gamma[0,1]$ is not contained in $\bigcup_{p\in\zIntegers^2}\ball[p][1/3]$, then $\gamma$ is contractible.
\end{lemma}
\begin{proof}
	Let $\pi:\hHalfplane\to\latticeComplement$ be a holomorphic covering map (in the sense of covering spaces), where $\hHalfplane$ is the upper half-plane.
	Let $z\in \latticeComplement$ and let $\widetilde z_0\in \pi^{-1}(z)$.
	Since $\pi$ is a covering map, there exists some minimal $\lambda_{\widetilde z_0}>0$ such that any path in $\hHalfplane$ from $\widetilde{z}_0$ to any other lift $\widetilde{z}_1\in\pi^{-1}(z)$ is of hyperbolic length at least $\lambda_{\widetilde z_0}$.
	Since deck transformations preserve hyperbolic lengths, we have that $\lambda_{\widetilde z_0}=\lambda_{\widetilde z_1}$ for any two lifts $\widetilde z_0,\widetilde z_1$ of $z$.
	Thus the map $\lambda:\latticeComplement\to\rReal$ given by $z\mapsto \lambda_{\widetilde z_0}$ is well defined and continuous.
	Thus a loop $\gamma$ passing through a point $z\in\latticeComplement$ is contractible if its hyperbolic length is less than $\lambda(z)$.

	The maps $z\mapsto z+1$ and $z\mapsto z+i$ are biholomorphic automorphisms of $\latticeComplement$ and hence they are also isometries.
	This implies that the map $\lambda$ is doubly periodic and so its values outside of $\bigcup_{p\in\zIntegers^2}\ball[p][1/3]$ are determined by its values on the compact set $[-1,1]^2\setminus \bigcup_{p\in\zIntegers^2}\ball[p][1/3]$.
	We may take the constant $\order{1}$ in the statement of the lemma to be any constant smaller than the minimal value of $\lambda$ on this compact set.
\end{proof}

\begin{definition}
	Let $\cell$ be a complex cell of length $\ell$.
	Each connected component $S$ of $\skeleton{\cell}$ is topologically a product of $\ell$ circles.
	Normalizing $\cell$ (as in \Cref{rem: normalized cells}) such that $S=\circle[1]^{\times \ell}$, we call the collection 
	\begin{equation}
		\qty{S_{\varZ,i}}_{i=1}^\ell=
		\Big\{\qty{\initial{\varZ}{1}{i-1}}\times\circle[1]\times\qty{\initial{\varZ}{i+1}{\ell}}\Big\}_{i=1}^\ell	
	\end{equation}
	the \emph{generating circles} of $S$ through $\varZ\in S$.
	For the sake of this definition, if $i> j$, we interpret $\qty{\initial{\varZ}{i}{j}}$ as the singleton $\cComplex^0$.
\end{definition}

\begin{proposition}\label{prop: short circles}
	Let $\hExt{\cell}{\rho}$ be a complex cell of length $\ell$, let $S_{\varZ,i}$ be a generating circle for some component $S$ of $\skeleton{\cell}$ and let $f:\hExt{\cell}{\rho}\to X$ be a holomorphic map where $X$ is a hyperbolic Riemann surface.
	Then the hyperbolic length of $f(S_{\varZ,i})$ is at most $\rho$.
\end{proposition}
\begin{proof}
	We may assume  $\cell$ is normalized such that $S=\circle[1]^{\times \ell}$.
	If $\ell=1$, this is a special case of~\cite[Fact 6]{BinyaminiNovikov2019}.
	We may reduce to the case $\ell=1$ by restricting $f$ to the intersection of $\hExt{\cell}{\rho}$ and the set $\qty{\initial{\varZ}{1}{i-1}}\times\cComplex\times\qty{\initial{\varZ}{i+1}{\ell}}$.
\end{proof}

Finally, we will prove the following lemma, which implies the fundamental lemma (\Cref{lem:fundamental lemma}) in the case $\epsilon=1$.
The case of general $\epsilon>0$ follows immediately by rescaling.
\begin{lemma}
	Let $\hExt{\cell}{\rho}$ be a complex cell of length $\ell$ where $\rho=\order{1}$ and let $f:\hExt{\cell}{\rho}\to\latticeComplement$ be holomorphic.
	Then either
	\begin{equation}\label{eq:fundamental lemma ball bound}
		f(\cell)\subset \ball[p][1/3]
	\end{equation}
	for some $p\in\zIntegers^2$, or
	\begin{equation}\label{eq:fundamental lemma hyperbolic metric bound}
		\diam{f(\cell)}[\latticeComplement]=\order[\ell]{\rho}.
	\end{equation}
	In either case, we have the following bound in the Euclidean metric:
	\begin{equation}\label{eq:fundamental lemma Euclidean metric bound}
		\diam{f(\cell)}[\cComplex]=\order[\ell]{1}.
	\end{equation}
\end{lemma}

\begin{proof}
	We may assume that the type of $\cell$ does not contain $\point$ and so $\dim\cell=\ell$.
	Let $\cell_n$ be the cell obtained from $\cell$ by replacing each coordinate of the form $\puncDisc[r]$ by $\annulus[r/n][r]$ and each coordinate of the form $\puncDisc*[r]$ by $\annulus[r][nr]$, where $n\in\nNatural$.
	Clearly, $\cell_n$ also admits a $\hyperbolicParameter{\rho}$-extension.
	Since \eqref{eq:fundamental lemma ball bound}, \eqref{eq:fundamental lemma hyperbolic metric bound} and \eqref{eq:fundamental lemma Euclidean metric bound} depend only on $\ell,\rho$ (and not on $n$), assuming the claim holds for $\cell_n$ instead of $\cell$ for all $n$ yields the result for $\cell$ as well.

	Thus, we may assume that the type of $\cell$ contains only $\disc,\annulus$.
	In particular, we have that the closure $\closure{\cell}$ of $\cell$ in $\cComplex^\ell$ is a compact subset of $\hExt{\cell}{\rho}$.

	We consider three cases.
	In the first two cases, we will show that $\pushforward{f}(\fundamentalGroup(\cell))$ is trivial in $\fundamentalGroup(\latticeComplement)$ and thus the claim follows by \Cref{prop: trivial fundamental group}.

	First, we assume that there exists some component $S$ of $\skeleton{\cell}$ such that $f(S)$ is not contained in $\bigcup_{p\in\zIntegers^2}\ball[p][1/3]$.
	Let $\varZ\in S$ be such that $f(z)\not\in \bigcup_{p\in\zIntegers^2}\ball[p][1/3]$.
	For $\rho=\order{1}$, \Cref{lem:shortest loop,prop: short circles} imply that every generating circle of $S$ through $\varZ$ is contractible.
	The inclusion $\iota:S\embeds\hExt{\cell}{\rho}$ induces a surjective map $\pushforward{\iota}:\fundamentalGroup(S)\to\fundamentalGroup\qty(\hExt{\cell}{\rho})$ and so $\pushforward{f}(\fundamentalGroup(\cell))$ is trivial in~$\fundamentalGroup(\latticeComplement)$.

	Next, we consider the case where there are two components $S_0$ and $S_1$ of $\skeleton{\cell}$ such that $f(S_0)\subset \ball[p_0][1/3]$ and $f(S_1)\subset \ball[p_1][1/3]$ for $p_0\ne p_1 \in \zIntegers^2$.
	For $i\in\qty{0,1}$, any element of $\pushforward{f}(\fundamentalGroup(S_i))$ is a power of the fundamental loop around $p_i$.
	Since every loop in $S_i$ is free-homotopy equivalent to a loop in $S_{1-i}$, we must have that $\pushforward{f}(\fundamentalGroup(S_i))$ is trivial and hence that $\pushforward{f}(\fundamentalGroup(\cell))$ is trivial.

	We are left with the case where $f(\skeleton{\cell})\subset \ball[p][1/3]$ for some $p\in\zIntegers^2$.
	Without loss of generality, we may assume $p=0$.
	Let $\varZ\in f(\closure{\cell})$ be of maximal absolute value.
	Then $\varZ\in\boundary f(\closure{\cell})\subset \boundary f(\cell)$.
	Therefore, by \Cref{prop: skeleton controls boundary}, $\varZ \in f(\skeleton{\cell})\subset \ball[0][1/3]$ and so $f(\cell)\subset \ball[0][1/3]$.

	We now explain how the bound~\eqref{eq:fundamental lemma Euclidean metric bound} on the Euclidean diameter of $f(\cell)$ follows from the bound~\eqref{eq:fundamental lemma hyperbolic metric bound} on the hyperbolic diameter of $f(\cell)$.
	
	Let $\lambda(z)\abs{\dd{z}}$ be the length element for the hyperbolic metric on $\latticeComplement$ at the point $z$.
	The map $\lambda(z)$ is continuous and satisfies $\lambda(z)=\lambda(z+1)=\lambda(z+i)$ for all $z\in\latticeComplement$.
	Therefore $\lambda(z)$ is determined by its values on $[-1,1]^2\setminus\zIntegers^2$.
	As $z$ tends to any point of $\zIntegers^2$, we have that $\lambda(z)\to\infty$ and so there exists some $0<r<1/3$ such that $\lambda(z)=\order*{1}$ when $z\in\bigcup_{p\in\zIntegers^2} \ball[p][r]$.
	We also have $\lambda(z)=\order*{1}$ for $z$ in the compact set $[-1,1]^2\setminus\bigcup_{p\in\zIntegers^2} \ball[p][r]$.
	Therefore $\diam{f(\cell)}[\latticeComplement]=\order[\ell]{\rho}$ implies $\diam{f(\cell)}[\cComplex]=\order[\ell]{\rho}$.
	Since $\rho=\order{1}$, we obtain the bound \eqref{eq:fundamental lemma Euclidean metric bound}.
\end{proof}

\section{\texorpdfstring{Proofs of the \s CPT and \s CPrT}{Proofs of the \#CPT and \#CPrT}}\label{sec: main theorems proofs}
In this section we prove \Cref{thm:cpt,thm:sharp preparation}.
We follow the outline described in \Cref{sec: proof outline}.

\subsection{\for{toc}{Step 1: $\text{\s CPT}_1$}\except{toc}{\texorpdfstring{Step 1: $(\text{\s CPT}_2)_{<\ell} \implies (\text{\s CPT}_1)_{\leq\ell}$}{Step 1}}}\label{sec: step 1 proof}
In this section, we consider analytic hypersurfaces $Z_1,\dots,Z_m\subset\hExt{\cell}{\rho}\odot\ext{\cell[F]}{\gamma}$ in a Weierstrass cell of length $\ell+1$ with gap $\gamma=1-1/\polyfd[\format*][\degree,M]$, such that $\hExt{\cell}{\rho}\odot\ext{\cell[F]}{\gamma},Z_1,\dots,Z_m\in\FD[\format][\degree]$.

If $\cell$ contains $\point$ in its type, then the result follows immediately from $(\text{\s CPT}_1)_{< \ell}$, thus we assume for the rest of the proof that this is not the case.
Let $\pi:\cComplex^{\ell+1}\to\cComplex^\ell$ be the projection given by $\pi(\initial{\varZ}{1}{\ell+1})=\initial{\varZ}{1}{\ell}$.

By first applying the sharp refinement theorem (\Cref{thm:sharp refinement}) to the base~$\cell$ and then pulling back by the resulting cellular cover, we may assume without loss of generality that $\hyperbolicParameter{\rho}=1/2$.

For every pair $Z_i, Z_j$, the union $Z_i\cup Z_j$ is an analytic hypersurface of $\ext{\cell}{1/2}\odot\ext{\cell[F]}{\gamma}$.
We first assume that the restriction of $\pi$ to this union is proper (in the case where $\cell[F]=\disc,\annulus$, this projection is indeed proper as $\closure{Z_i}\cap(\ext{\cell}{1/2}\odot\boundary(\ext{\cell[F]}{\gamma}))=\emptyset$.
See \Cref{sec: proper projections puncture} for the case $\cell[F]=\puncDisc,\puncDisc*$, where the last coordinate of ${Z_i}$ may approach the puncture).
Thus, as in~\cite[Lemma 70]{BinyaminiNovikov2019}, the projection $\pi_{i,j}=\restrict{\pi}{Z_i\cup Z_j}:Z_i\cup Z_j \to \ext{\cell}{1/2}$ is a $\polyfld$-sheeted covering map (in the sense of covering spaces) outside the zero-set of the discriminant $\mathcal{D}_{i,j}$ of the following monic polynomial in $w$:
\begin{equation}\label{eq: interpolating polynomial}
	P_{i,j}(\varZ,w):=\prod_{\eta\in\pi_{i,j}^{-1}(\varZ)}(w-\eta)\qc \varZ\in\ext{\cell}{1/2}.
\end{equation}
The coefficients of $P_{i,j}$, as a polynomial in $w$, continue to holomorphic  functions on $\ext{\cell}{1/2}$, hence $\qty{\mathcal{D}_{i,j}=0}$ is an analytic hypersurface.

This description of $\qty{\mathcal{D}_{i,j}=0}$ is, however, not good enough for our purposes in terms of complexity, since it involves the zero-set of a polynomial in $\polyfld$ (rather than $\order[\format*]{1}$) variables, corresponding to the number of preimages $\eta\in\pi_{i,j}^{-1}(\varZ)$.
Instead, we note that the zero-set of this discriminant is in $\FD*[\format*][\degree]$ since it may also be described as the complement to the set of points ${\varZ\in \pi_{i,j}(Z_i\cup Z_j)}$ for which $\pi_{i,j}$ is locally bijective near each preimage of~$\varZ$.

Since the length of $\cell$ is smaller than that of $\cell\odot\cell[F]$, we may inductively apply $(\text{\s CPT}_2)_{<\ell}$ to $\cell$ and the sets $\qty{\mathcal{D}_{i,j}=0}$ for $1\leq i,j\leq M$.
Let $\{f_k:\ext{\cell_k}{1/2}\to\ext{\cell}{1/2}\}$ denote the resulting cellular cover.
This cover is of size $\polyfd[\format*][\degree, M]$ and each function $f_k$ is in $\FD*[\format*][\degree]$.

If $f_k(\ext{\cell_k}{1/2})$ is contained in one the sets $\qty{\mathcal{D}_{i,j}=0}$, then, since cellular maps preserve dimension, we have $\dim \cell_k < \dim \cell$ and so $\dim (\cell_k \odot \pullback{f}_k\cell[F]) < \dim (\cell\odot\cell[F])$.
In this case, we may proceed inductively by applying $\text{\s CPT}_2$ with respect to the sets ${(f_k\odot \idmap)}^{-1}(Z_i)$ (see \Cref{def: odot id} for the notation $\pullback{f}_k\cell[F]$ and $f_k\odot \idmap$).

It remains to consider the case where $f_k(\ext{\cell_k}{1/2})$ is disjoint from all the sets $\qty{\mathcal{D}_{i,j}=0}$.
Pulling back by $(f_k\odot\idmap)$ and renaming, we may assume that $\pi$ restricted to $Z_1\cup\cdots\cup Z_M$ is a proper unramified covering map --- indeed, any ramification point is contained in the ramification locus of some pair $Z_i \cup Z_j$.

We may now proceed as in the proof the CPT in~\cite[Section 8.1]{BinyaminiNovikov2019}, using the notation of~\cite{BinyaminiNovikov2019}.
We apply the clustering constructions of~\cite[Section 6.3]{BinyaminiNovikov2019} with gap parameter $\gamma$ to the sections of the restriction of $\pi$ to the union ${Z_1\cup\cdots\cup Z_M}$ (instead of a single given surface as in~\cite{BinyaminiNovikov2019}).
Since the number of sections is $\polyfd[\format*][\degree,M]$, our assumption that $\gamma=1-1/\polyfd[\format*][\degree,M]$ is sufficient in order to carry out the clustering constructions with the desired bounds on the sizes and complexities of the resulting cellular covers.

Furthermore, since the restriction of $\pi$ to $Z_1\cup\dots\cup Z_M$ is unramified, each of the sections of this projection is contained in only one of the sets $Z_i$.
Thus, by \Cref{rem: complexity of component}, these sections (as well as their univalued pullbacks to the appropriate $\nu$-cover, see~\cite[Definition 23]{BinyaminiNovikov2019}) are in $\FD*[\format*][\degree]$ and so the resulting cells and cellular maps are also of the desired complexity.
Here we use the assumption that $\qty{\FD[\format][\degree]}$ has \s CD, noting that each of the sections of $\pi$ as above is a connected component of a set in~$\FD*[\format*][\degree]$.

\subsubsection{Proper projections near a puncture}
\label{sec: proper projections puncture}

Continuing with the same notation as in \Cref{sec: step 1 proof} and assuming $\cell[F]=\puncDisc,\puncDisc*$, it remains to explain why we may reduce to the case where the projections $\pi:Z_i\to\ext{\cell}{1/2}$ are proper.
We may assume $\cell[F]=\puncDisc[1]$ by inverting and rescaling the last coordinate (see \Cref{rem: normalized cells}).

We will use the following ``definable Remmert--Stein'' theorem.
We note that, in comparison with the classical formulation of the theorem, there is no assumption regarding the dimensions of the relevant sets.
\begin{theorem}[{\cite[Theorem 4.13]{PeterzilStarchenko2008}}]
\label{thm: Remmert--Stein}
Let $U\subset\cComplex^\ell$ be a definable domain.
Let $A\subset U$ and $Z\subset U\setminus A$ be relatively closed subsets which are locally given as the zero-sets of definable holomorphic functions.
Then the closure of $Z$ in $U$ is an analytic subset.
\end{theorem}

In order to apply \Cref{thm: Remmert--Stein}, we need the following lemma.
\begin{lemma}
	\label{lem: definable hypersurface}
	Let $U\subset \cComplex^{\ell+1}$ be a domain and let $Z\subset U$ be a definable analytic hypersurface.
	Then $Z$ is locally given as the zero-set of a definable holomorphic function.
\end{lemma}
\begin{proof}
	Up to a unitary change of coordinates (which is, in particular, semialgebraic and hence definable), we may locally express $Z$ as the zero-set of a Weierstrass polynomial $P(\initial{\varZ}{1}{\ell};\varZ_{\ell+1})$ of some degree $k$ with respect to $\varZ_{\ell+1}$, defined in a region $V\times\cComplex$ for some open polydisc $V\subset\cComplex^\ell$ (see~\cite[Theorem 2.8]{Chirka1989}).
	We will show that the coefficients of $P(\initial{\varZ}{1}{\ell};\varZ_{\ell+1})$ are definable functions of $\initial{\varZ}{1}{\ell}\in V$.

	Consider the following definable subset of $\cComplex^{\ell+k}$, with coordinates $(\initial{\varZ}{1}{\ell},\initial{\varW}{1}{k})$:
	\begin{equation}
		\label{eq: k section space}
		\qty(\bigcap_{i=1}^k \qty{(\initial{\varZ}{1}{\ell},\varW_i)\in Z})
		\setminus \qty(\bigcup_{i\neq j}\qty{\varW_i=\varW_j}).
	\end{equation}
	The graphs of the coefficients of $P(\initial{\varZ}{1}{\ell};\varZ_{\ell+1})$, as functions of $\initial{\varZ}{1}{\ell}\in V$, are given by applying elementary symmetric polynomials to the last $k$ coordinates of \eqref{eq: k section space}, taking the topological closure, and then intersecting with $V\times\cComplex$.
\end{proof}

Let $\{Z_i\}$ be the analytic hypersurfaces of $\hExt{\cell}{\rho}\odot\ext{\cell[F]}{\gamma}$ as in \Cref{sec: step 1 proof}.
Considering $\ext{\cell}{1/2}\odot\puncDisc[1]$ as the set difference of $\ext{\cell}{1/2}\odot\disc[1]$ and $\ext{\cell}{1/2}\odot\point$, \Cref{lem: definable hypersurface,thm: Remmert--Stein} imply that the closure of each $Z_i$ in $\ext{\cell}{1/2}\odot\disc[1]$ is an analytic subset.
It is of the same dimension as $Z_i$ by~\cite[Theorem 3.22 and Proposition 3.17]{Coste1999}.
These closures are all in $\FD*[\format*][\degree]$.

We may now repeat the arguments of \Cref{sec: step 1 proof} with $\cell[F]$ replaced by $\disc[1]$ and each $Z_i$ replaced by its appropriate closure.
Without loss of generality, we may assume that one of the $Z_i$ is equal to $\ext{\cell}{1/2}\odot\point$.
After reducing to the case where the projection $\pi:\bigcup Z_i\to\ext{\cell}{1/2}$ is a proper unramified cover, we may return (by rescaling and inverting as needed) to the original value of the fiber~$\cell[F]$.
In this case the projection $\pi:\bigcup Z_i\to\ext{\cell}{1/2}$ is still a proper unramified cover and we may proceed with the clustering constructions of~\cite[Section 6.3]{BinyaminiNovikov2019}.

\subsection{\for{toc}{Step 2: $\text{\s CPT}_2$}\except{toc}{\texorpdfstring{Step 2: $(\text{\s CPT}_2)_{<\ell} + (\text{\s CPT}_1)_{\leq\ell}\implies (\text{\s CPT}_2)_{\leq\ell}$}{Step 2}}}\label{sec: step 2 proof}

This is the main part of the argument.
We fix the following notation and assumptions for the rest of this section.
Let $Z_1,\dots,Z_M$ be analytic hypersurfaces of a complex cell $\hExt*{\cell\odot\cell[F]}{\rho}$ of length $\ell+1$ and assume that $\hExt*{\cell\odot\cell[F]}{\rho},Z_1,\dots,Z_m\in\FD[\format][\degree]$.

If $\cell\odot\cell[F]$ contains $\point$ in its type, then the result follows  immediately from $(\text{\s CPT}_2)_{< \ell}$.
Thus, we may assume for the rest of the proof that this is not the case.

We will give the details in the case $\cell[F]=\annulus$, the other cases are similar.
Rescaling $\cell[F]$ by its outer radius, we may assume that $\cell[F]=\annulus[r][1]$, where $r$ is a holomorphic function on $\hExt{\cell}{\rho}$ satisfying $0<\abs{r(\varZ)}<1$ for all $\varZ\in\hExt{\cell}{\rho}$.

Let $\projfiber$ and $\projbase$ be the projection maps given by $\projfiber(\initial{\varZ}{1}{\ell+1})=z_{\ell+1}$ and $\projbase(\initial{\varZ}{1}{\ell+1})=\initial{\varZ}{1}{\ell}$, for $\initial{\varZ}{1}{\ell+1}\in\cComplex^{\ell+1}$.
We write $\operatorname{abs}:\cComplex\to\rReal$ for the absolute value map.

By first applying the sharp refinement theorem and pulling back by the resulting cellular maps, we may assume that $\hyperbolicParameter{\rho}$ is equal to some constant.
For convenience, we will take $\hyperbolicParameter{\rho}=1/6$.

\subsubsection{\texorpdfstring{Removing an $\epsilon$-net}{Removing an ε-net}}\label{sec: removing an epsilon net}
Let $0<\epsilon<1$, whose value will be determined later.
In this section, we reduce to the case where each of the projections $\pi_{\ell+1}(Z_i)$ does not intersect the set $\epsilon\zIntegers^2$  in $\annulus[1][6]$, i.e.\ near the outer, normalized boundary component of $\cell[F]$.
To treat the other boundary component, we may invert the last coordinate, normalize the (new) outer radius of $\cell[F]$ and repeat the arguments given below before applying another inversion and normalization to return to the original coordinates.

Let $p\in\epsilon\zIntegers^2 \cap \annulus[1][6]$ and let $\Gamma_p$ be the graph of the constant function of value $p$ over $\ext{\cell}{1/6}$.
The intersections $Z_i \cap \Gamma_p$ are all analytic and, up to a small perturbation of $\epsilon$, are either empty or of complex codimension $2$ in $\ext*{\cell\odot\cell[F]}{1/6}$.
The restriction of $\projbase$ to each $Z_i \cap \Gamma_p$ is proper and so the image $\projbase(Z_i \cap \Gamma_p)$, if not empty, is an analytic subset of complex codimension $1$ in $\ext{\cell}{1/6}$ (this may also be seen, for example, by plugging in $p$ in the last variable of the local defining functions of $Z_i$).

We apply $(\text{\s CPT}_2)_{<\ell}$ to $\ext{\cell}{1/6}$ to obtain a cellular cover by cells admitting $1/6$-extensions which is compatible with all the sets $\projbase(Z_i \cap \Gamma_p)$.
The size of this cover is $\polyfd[\format*][\degree,\epsilon^{-1}, M]$ and the corresponding cells and cellular maps are in $\FD*[\format*][\degree]$ (crucially, their complexity does not depend on $\epsilon$).
The dependence on $\epsilon$ in the size of the cover follows from the fact that there are $\poly(\epsilon^{-1})$-many possible values of $p\in\epsilon\zIntegers^2 \cap \annulus[1][6]$.

The cells of this cover are either of smaller dimension than that of $\cell$, in which case we proceed by induction on dimension as in Step 1; or else their images are disjoint from the sets $\projbase(Z_i \cap \Gamma_p)$.

Pulling back by a cellular map in this cover, we may assume that each $\projfiber(Z_i)$ does not intersect $\epsilon\zIntegers^2 \cap \annulus[1][6]$ as long as we eventually choose $\epsilon$ such that ${\epsilon^{-1}=\poly_{\format*}(D,M)}$. 
This constraint is needed in order to get the correct bounds on the size of the resulting cellular cover.

\subsubsection{Radial bound}
\label{sec: RB}
We consider the following bound on the diameter of an analytic hypersurface in the radial direction of the last coordinate.

\begin{proposition}[Radial bound, RB]
	\label{prop: RB}
	Let $\ext{\cell}{1/6}$ be a cell of length $\ell$ whose type does not contain~$\point$.
	Let $Z\subset\ext{\cell}{1/6}\odot\annulus[1][6]$ be an analytic hypersurface such that $\ext{\cell}{1/6},Z\in\FD[\format][\degree]$ and such that the projection $\projfiber(Z)$ does not intersect $\epsilon\zIntegers^2$.
	Then for each connected component $Z_0$ of $Z\cap\qty(\ext{\cell}{1/3}\odot\annulus[2][3])$, we have
	\begin{equation}
		\label{eq: radial bound}
		\diam{\operatorname{abs}(\projfiber(Z_0))}[\rReal]<\epsilon\cdot\polyfld.
	\end{equation}
\end{proposition}

We first show how to use \eqref{eq: radial bound} to finish the proof of $\text{\s CPT}_2$ and then outline the proof of \Cref{prop: RB}, which is carried out in the subsequent~sections.

\begin{proof}[Proof of $\text{\s CPT}_2$ from \Cref{prop: RB}]
	Let $\Sigma\subset\rReal$ be given by
	\begin{equation}
		\Sigma = \operatorname{abs}\bigg({\projfiber}\Big(\big(\bigcup_{i}Z_{i}\big)\cap(\ext{\cell}{1/3}\odot\annulus[2][3])\Big)\bigg).
	\end{equation}
	For each $i$, the number of connected components of the intersection of $Z_i$ with the cell $\ext{\cell}{1/3}\odot\annulus[2][3]$ is $N=\polyfld$.
	Therefore, by \eqref{eq: radial bound}, we have that $\Sigma$ is the union of at most $NM$ intervals contained in the interval $(2,3)$, whose total length is at most $\epsilon N M \polyfld$.
	For $\epsilon^{-1}=\poly_{\format*}(D,M)$, this total length is less than $1/3$.

	In this case it is easy to see that there exists some constant $r_1\in(2,3)$ such that the interval $(r_1,r_1+\frac{1}{3NM})\subset(2,3)$  does not intersect $\Sigma$.

	Let
	\begin{equation}
		\gamma 	= \frac{r_1}{r_1+\frac{1}{3NM}}
				= 1 - \frac{1}{\polyfd[\format*][\degree,M]},
		\end{equation} 
	then $\ext{\cell}{1/3}\odot\annulus[r_1][r_1/\gamma]$ does not intersect $\bigcup_i Z_i$.

	By inverting the last coordinate and repeating this construction, we may assume that $\gamma$ is such that there exists a constant $r_0\in(1/3,1/2)$ for which $\ext{\cell}{1/3}\odot\annulus[\gamma\cdot r r_0 ][r r_0 ]$ does not intersect $\bigcup_i Z_i$.
	We may now apply $\text{\s CPT}_1$ to the cell $\ext{\cell}{1/3}\odot\ext{\annulus[rr_0][r_1]}{\gamma}$ to finish the proof of $\text{\s CPT}_2$.
\end{proof}

For the proof of \Cref{prop: RB}, we consider the following two variants.
\begin{enumerate}[leftmargin=3\parindent,align=right]
	\item[$(\text{RB}_1)$] \Cref{prop: RB} when the type of $\cell$ contains only $\disc$.
	\item[$(\text{RB}_2)$] \Cref{prop: RB} when the type of $\cell$ contains only $\disc, \annulus$.
\end{enumerate}
In the subsequent sections we will prove that $\text{RB}_{2}$ implies \Cref{prop: RB}, that $\text{RB}_{1}$ implies $\text{RB}_{2}$ and, finally, establish $\text{RB}_{1}$.

\subsubsection{\texorpdfstring{$\text{RB}_2\implies \text{RB}$}{RB2 implies RB}}

In this section, we reduce \Cref{prop: RB} to the case where the type of $\cell$ does not contain $\puncDisc$ or $\puncDisc*$.
Let $\cell$ be as in \Cref{prop: RB} and let $\cell_n$ be the cell obtained from $\cell$ by replacing each coordinate of the form $\puncDisc[r]$ by $\annulus[r/n][r]$ and each coordinate of the form $\puncDisc*[r]$ by $\annulus[r][nr]$, where $n\in\nNatural$.
Clearly, $\cell_n$ also admits a $1/6$-extension which is in $\FD*[\format*][\degree]$ (in particular, its complexity does not depend on $n$.
Cf.\ \Cref{rem: extension complexity}).

Let $c:[0,1]\to Z_0$ be a continuous path.
Since $c([0,1])$ is a compact subset of $\ext{\cell}{1/3}\odot\annulus[2][3]$, it is contained in $\ext{\cell_n}{1/3}\odot\annulus[2][3]$ for all sufficiently large $n$.
Fixing such an $n$, let $Z_{0}'$ be a connected component of ${Z\cap(\ext{\cell_n}{1/3}\odot\annulus[2][3]})$ containing $c([0,1])$.
Assuming that \Cref{prop: RB} holds for $\cell_n$ in place of $\cell$, we have that the distance between $(\operatorname{abs}\comp\,\projfiber\comp c)(0)$ and $(\operatorname{abs}\comp\,\projfiber\comp c)(1)$ is at most $\diam{\operatorname{abs}(\projfiber(Z_0'))}[\rReal]$ and hence bounded by $\epsilon\cdot\polyfld$.
We deduce \Cref{prop: RB} for the original cell $\cell$ by taking the supremum over all such paths~$c$.

\subsubsection{\texorpdfstring{$\text{RB}_1\implies \text{RB}_2$}{RB1 implies RB2}}
\label{sec: reduce to only discs}
Assuming that the type of $\cell$ contains only $\disc,\annulus$, we first show the following lemma, similar to \Cref{prop: skeleton controls boundary}.
See \Cref{def: skeleton} for the notion of a skeleton.
\begin{lemma}
	\label{lem: skeleton controls boundary analytic set}
	Let $w \in \boundary(\projfiber(\closure{Z_0}))$, where $\closure{Z_0}$ is the closure of $Z_0$ in $\cComplex^{\ell+1}$.
	Then there is $\varZ \in \closure{Z_0}$ such that $\varZ_{\ell+1}=w$ and such that $\initial{\varZ}{1}{\ell}\in \skeleton{\ext{\cell}{1/3}}$, unless $\abs{w}=2,3$.
\end{lemma}
\begin{proof}
	The assumption on the type of $\cell$ implies that the closure of $\ext{\cell}{1/3}\odot\annulus[2][3]$ in $\cComplex^{\ell+1}$ is compactly contained in $\ext{\cell}{1/6}\odot\annulus[1][6]$, and hence so is $\closure{Z_0}$.
	Therefore there exists $\varZ\in \closure{Z_0}$ such that $\varZ_{\ell+1}=w$.
	If $\initial{\varZ}{1}{\ell}\in \skeleton{\ext{\cell}{1/3}}$, we are done.
	Otherwise, there exists some coordinate $1\leq j\leq \ell$ such that
	\begin{equation}
		\abs{\varZ_j}<3\abs{r_j(\initial{\varZ}{1}{j-1})}
	\end{equation}
	when the $j$-th coordinate of $\cell$ is a disc, or
	\begin{equation}
		\frac{1}{3}\abs{r_{j,1}(\initial{\varZ}{1}{j-1})}<\abs{\varZ_j}<3\abs{r_{j,2}(\initial{\varZ}{1}{j-1})}
	\end{equation}
	when the $j$-th coordinate of $\cell$ is an annulus.
	We consider the maximal such $j$ for this choice of $\varZ$, meaning that $\varZ_{j+1},\dots,\varZ_{\ell}$ all lie on boundary components of the corresponding coordinates of $\ext{\cell}{1/3}$.
	By normalizing $\cell$, we may assume that the radii determining these boundary components are all constant.
	Furthermore, we choose $\varZ$ such that $j$ is minimal over all possible choices of $\varZ$.
	
	Let $\cell[F]_j=\cell[F]_j(\initial{\varZ}{1}{j-1})$ be the first coordinate of the cell $\cell_{\initial{\varZ}{1}{j-1}}\subset\cComplex^{\ell-j+1}$ (see \Cref{def: initial and fiber notation} for this notation).
	By minimality of $j$, we have that $\varZ$ is an isolated point of
	\begin{equation}
		\closure{Z_0}\cap \qty(\qty{\initial{\varZ}{1}{j-1}}\times\ext{\cell[F]_j}{1/3}\times\qty{\initial{\varZ}{j+1}{\ell}}\times\qty{w}).
	\end{equation}

	If $2<\abs{w}<3$, \Rouche's theorem implies that for any $w'$ sufficiently close to $w$ there is $z_j'\in \ext{\cell[F]_j}{1/3}$ such that 
	$\qty{\initial{\varZ}{1}{j-1}}\times\{z_j'\}\times\qty{\initial{\varZ}{j+1}{\ell}}\times\qty{w'}\subset\closure{Z_0}$ and hence $w'\in\projfiber(\closure{Z_0})$, contradicting our assumption that $w\in \boundary(\projfiber(\closure{Z_0}))$.
\end{proof}

We now use \Cref{lem: skeleton controls boundary analytic set} to reduce \Cref{prop: RB} to the case where the type of $\cell$ contains only $\disc$ (see \Cref{fig: RB}).
Let $\cell,Z,Z_0$ be as in \Cref{prop: RB} and assume that the type of $\cell$ contains only $\disc$ and $\annulus$.
Denote by $\circle[r]$ the circle of radius $r$ around the origin.
For $x\in\operatorname{abs}({\projfiber(\closure{Z_0})})$, the preimage $\operatorname{abs}^{-1}(x)$ in $\projfiber(\closure{Z_0})$ is either the circle $\circle[\abs{x}]$  or else it intersects $\boundary \projfiber(\closure{{Z_0}})$.

Let $\{\Gamma_i\}$ be the collection of connected components of $\boundary\projfiber(\closure{{Z_0}})$ and let $\{\Gamma^*_i\}$ be the collection of connected components of the set $\{z\in \cComplex\st \circle[\abs{z}]\subset \projfiber(\closure{Z_0})\}$.
Each of these collections is of size $\polyfld$.

The connected set $\operatorname{abs}(\projfiber(\closure{{Z_0}}))$ is covered by the sets$\{\operatorname{abs}(\Gamma_i)\}$ and $\{\operatorname{abs}(\Gamma^*_i)\}$, and so $\diam{\operatorname{abs}({\projfiber(Z_0)})}[\rReal]$ is at most
\begin{equation}
	\label{eq: radial bound gammas}
	\sum_i \diam{\operatorname{abs}({\Gamma_i})}[\rReal]+\sum_i \diam{\operatorname{abs}({\Gamma^*_i})}[\rReal].
\end{equation}
Since $\projfiber(Z_0)\cap \epsilon\zIntegers^2=\emptyset$, we have $\diam{\operatorname{abs}({\Gamma^*_i})}[\rReal]<\epsilon$ for all $i$, and so \eqref{eq: radial bound gammas} is at most
\begin{equation}
	\sum_i \diam{\operatorname{abs}({\Gamma_i})}[\rReal]+\epsilon\cdot\polyfd[\format,\ell][\degree].
\end{equation}
Hence it remains to show that the diameter of each of the sets ${\operatorname{abs}({\Gamma_i})}$ is bounded by $\epsilon\cdot\polyfld$.

As in the proof of the sharp refinement theorem (\Cref{thm:sharp refinement}), we may cover $\skeleton{\ext{\cell}{1/3}}$ by $\order[\ell]{1}$ cellular polydiscs $\{\polydisc_j\}$ whose $1/6$-extensions lie in $\ext{\cell}{1/6}$.
Let $\{Y_{j,k}\}$ be the $\polyfld$ connected components of $Z\cap(\ext{\polydisc_j}{1/3}\odot\annulus[2][3])$.
By \Cref{lem: skeleton controls boundary analytic set}, each component $\Gamma_i$ is covered by the projections $\{\projfiber(\closure{ Y_{j,k}})\}$ and by the circles of radius $2$ and $3$.
These circles do not contribute to the diameter of $\operatorname{abs}({\Gamma_i})$, however.
Assuming \Cref{prop: RB} holds with $\cell$ replaced by each of the cellular polydiscs $\polydisc_j$ and summing~\eqref{eq: radial bound} over all $j,k$, we get that $\diam{\operatorname{abs}({\Gamma_i})}[\rReal]<\epsilon\cdot\polyfld$, as required.

\begin{figure}
	\includegraphics[width=0.66\linewidth]{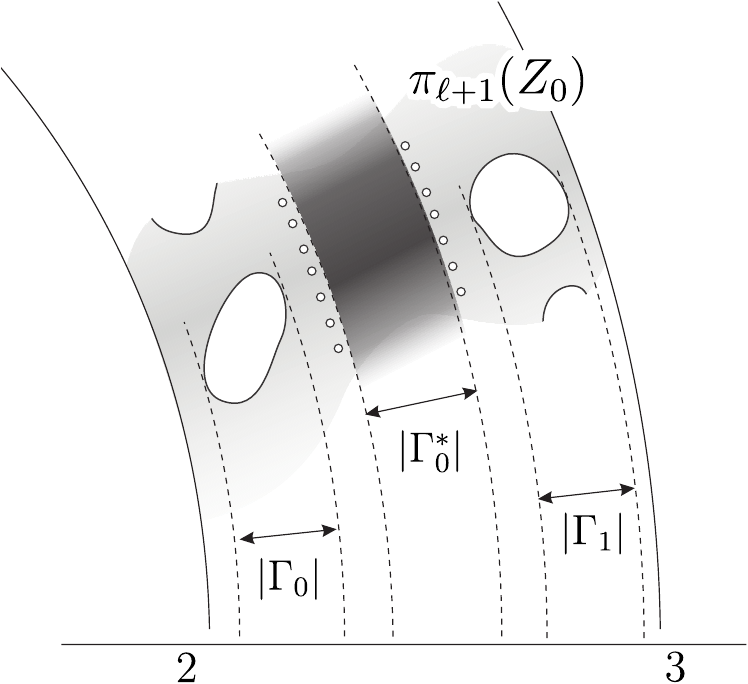}
	\caption{\emph{Bounding $\diam{\operatorname{abs}({\projfiber(Z_0)})}[\rReal]$ when the type of $\cell$ contains annuli.}
	All change in the radial direction of $\projfiber(Z_0)$ is accounted for by either a connected component $\Gamma_i$ of the boundary of $\projfiber(Z_0)$ or by an annulus $\Gamma^*_i$ which is entirely contained in $\projfiber(Z_0)$.
	The boundary of $\projfiber(Z_0)$ has few connected components, each of which has small radial diameter by the case where $\cell$ is a cellular polydisc.
	There are also few such maximal annuli $\Gamma^*_i$ and, since they cannot contain any point of an $\epsilon$-net, their width is small.
	}
	\label{fig: RB}
\end{figure}

\subsubsection{Proof of \texorpdfstring{$\text{RB}_1$}{RB1}}

It remains to treat the case where the type of $\cell$ contains only discs.
We may normalize $\cell$ and assume $\cell=\disc[1]^{\times \ell}$ (see \Cref{rem: normalized cells}).
The intersection of a Euclidean ball in $\cComplex^{\ell+1}$ and any fixed translation of $Z_0$ is in $\FD[\format'][\degree']$ where $\format'=\order[\format*]{1}$ and $\degree'=\polyfld$.
Let $\gamma<1-1/\polyfd[\format',\ell][\degree']$ be as in \Cref{prop: Weierstrass polydiscs}.
We may take $\gamma$ to be larger than $1/2$.
The bounds $\order*[\ell]{1}<\gamma<1-1/\polyfld$ imply that the hyperbolic extension parameters corresponding to an $\gamma$-extension and to an $(1-\gamma)$-extension, as well as the inverses of these hyperbolic extension parameters, are all at most $\polyfld$.

We cover ${\ext{\cell}{1/3}\odot\annulus[2][3]}$ by a collection $\mathcal{B}$ of $\polyfld$ Euclidean balls whose $(1-\gamma)^2$-extensions lie in $\ext{\cell}{1/6}\odot\annulus[1][6]$.
Let $\ball\in\mathcal{B}$ be one of these balls and let $\widetilde Z$ be the intersection of $Z_0$ with $\ext{\ball}{(1-\gamma)^2}$.
Without loss of generality and to ease notation, we assume that the center of $\ball$ is at the origin.

We apply \Cref{prop: Weierstrass polydiscs} to $\widetilde{Z}$ and $\ext{\ball}{(1-\gamma)^2}$ and denote by $\polydisc$ and $U$ the resulting polydisc and unitary transformation.
Write $\polydisc=\polydisc*[1..\ell][\ell+1]$, where $\polydisc[1..\ell]\subset\cComplex^\ell$ and $\polydisc[\ell+1]\subset\cComplex$.
By construction, $\polydisc*[1..\ell][\ell+1]$ is a Weierstrass cell of gap $\gamma$ for $U^{-1}(\widetilde{Z})$ (see \Cref{def: weierstrass cell}).
We also have
\begin{equation}
	\label{eq: polydisc contained in ball}
	\ext{\ball}{1-\gamma}\subset\polydisc\subset \ext{\ball}{(1-\gamma)^2}
\end{equation}
and hence
\begin{equation}
	\ball\subset \ext{\polydisc}{1/(1-\gamma)}\subset \ext{\ball}{1-\gamma}.
\end{equation}

Since $\gamma>1/2$, we have $\gamma>1-\gamma$.
Therefore,
\begin{equation}
	\label{eq: ball contained in polydisc}
	\ball \subset \qty(\ext{\polydisc[1..\ell]}{1/(1-\gamma)}\times\ext{\polydisc[\ell+1]}{1/\gamma}) \subset \polydisc.
\end{equation}

We now apply $(\text{\s CPT}_1)_{\leq \ell}$ to $\polydisc$ and $U^{-1}(\widetilde{Z})$, taking $\polydisc$ to be the $(1-\gamma,\gamma)$ extension of $\ext{\polydisc[1..\ell]}{1/(1-\gamma)}\times\ext{\polydisc[\ell+1]}{1/\gamma}$ in this application.
Let $\{f_j:\hExt{\cell_j}{\rho}\to\polydisc\}$ be the resulting cellular cover, where $\rho=\order{1}$ is chosen small enough such that the conditions of the fundamental lemma (\Cref{lem:fundamental lemma}) are satisfied for $\hExt{\cell_j}{\rho}$.
We can also choose this $\rho$ to be $\order*{1}$ and so the size of this cover is $\polyfld$.

We consider from now only those indices $j$ such that the image of $f_j$ intersects $U^{-1}(\widetilde{Z})$.
By \eqref{eq: polydisc contained in ball} and the construction of $B$, the images of the cells $\hExt{\cell_j}{\rho}$ under the holomorphic maps $\projfiber\comp U \comp f_j$ lie in $\latticeComplement[\epsilon]$.

By the fundamental lemma (\Cref{lem:fundamental lemma}), we have
\begin{equation}
	\label{eq: radial bound image of cell}
	\diam{\projfiber\comp U\comp f_j ({\cell}_j)}[\cComplex]<\order[\ell]{\epsilon}.
\end{equation}
The sets $\{f_j(\cell_j)\}$ cover $U^{-1}(\widetilde{Z})\cap\qty(\ext{\polydisc[1..\ell]}{1/(1-\gamma)}\times\ext{\polydisc[\ell+1]}{1/\gamma})$ and hence also cover $U^{-1}(\widetilde{Z})\cap B=U^{-1}(Z_0\cap B)$ by \eqref{eq: ball contained in polydisc}.
Let $\{Y_k\}$ be the $\polyfld$ connected components of $Z_0\cap B$.
Summing \eqref{eq: radial bound image of cell} over $j$ yields
\begin{equation}
	\label{eq: radial bound component of ball intersection}
	\diam{\projfiber(Y_k)}[\cComplex]<\epsilon\cdot\polyfld.
\end{equation}
In the same manner, since $Z_0$ is connected, we may sum \eqref{eq: radial bound component of ball intersection} over all $k$ and all balls in $\mathcal{B}$ to get
\begin{equation}
	\diam{\projfiber(Z_0)}[\cComplex]<\epsilon \cdot \polyfld
\end{equation}
and hence also
\begin{equation}
	\diam{\operatorname{abs}({\projfiber(Z_0)})}[\rReal]<\epsilon\cdot \polyfld
\end{equation}
as required.

\subsection{Proof of the \texorpdfstring{\s CPrT}{\#CPrT}}
\label{sec: CPrT}
In this section, we prove the \s CPrT (\Cref{thm:sharp preparation}), assuming the ``weak version'' of the \s CPT (where we do not assume the resulting cellular maps are prepared in the sense of \Cref{def: prepared map}) which we proved in \Cref{sec: step 1 proof,sec: step 2 proof}, and thus conclude the proof of the \s CPT as stated in \Cref{thm:cpt}.
The proof is similar to that in~\cite[Section 8.2]{BinyaminiNovikov2019}, but is somewhat more streamlined and resolves a small technical issue in the original --- the main differences are the different definition of cellular map (\Cref{def: cellular map}) that we use and replacing the use of \cite[Theorem 10, WPT]{BinyaminiNovikov2019} by an application of the \s CPT.

By induction on the length of the cells as in ~\cite{BinyaminiNovikov2019}, it is enough to prove the following lemma (cf.\ \cite[Lemma 76]{BinyaminiNovikov2019}).
\begin{lemma}
	Let $f:\hExt*{\cell\odot\cell[F]}{\rho}\to\cComplex^{\ell+1}$ be a cellular map in $\FD[\format][\degree]$.
	Then there exists a cellular cover $\{g_j:\hExt{\cell_j}{\sigma}\to\hExt*{\cell\odot\cell[F]}{\rho}\}$ of size $\polyfd[\format*][\degree,\rho,1/\sigma]$ such that each $g_j$ is in $\FD*[\format*][\degree]$ and each $f\comp g_j$ is prepared in the last variable.
\end{lemma}
\begin{proof}
	We first note that if $\cell[F]=\point$, then the map $f$ is already prepared in the last variable, and so we assume from now on that $\cell[F]$ is of type $\disc$, $\puncDisc$, $\puncDisc*$ or $\annulus$.
	
	Let $\widetilde f$ denote the last coordinate of $f$.
	Assuming $\cell\odot\cell[F]$ is of length $\ell+1$ with coordinates $\initial{\varZ}{1}{\ell+1}$, we consider the cell $\widetilde \cell$ given by $\cell\odot\cComplex\odot\cell[F]$ (see \Cref{rem: complex basic fiber} for this notation), which also admits a $\hyperbolicParameter{\rho}$-extension, and denote its coordinates by $\initial{\varZ}{1}{\ell},w,\varZ_{\ell+1}$.
	
	Let $Z\subset{\hExt{\widetilde \cell }{\rho}}$ be given by $\{w=\widetilde f(\initial{\varZ}{1}{\ell+1})\}$.
	We apply the \s CPT to $\hExt{\widetilde \cell }{\rho}$ and $Z$ to obtain a cellular cover $\{g_j:\hExt*{\cell_j\odot\cell[F]_j}{\sigma}\to\hExt{\widetilde \cell}{\rho}\}$ compatible with $Z$.

	Let $g_j$ be one of the maps in this cover such that $g_j(\cell_j\odot \cell[F]_j)\subset Z$.
	By \Cref{def: cellular map}, we have that $\pdv{\varZ_{\ell+1}}\widetilde{f}\neq 0$ and so the fibers of the natural projection $Z\to \initial{\hExt{\widetilde{\cell}}{\rho}}{1}{\ell+1}$ are all zero-dimensional.
	Therefore the fiber $\cell[F]_j$ must be of type $\point$.
	We denote the coordinates of $\cell_j\odot\cell[F]_j$ by $\initial{\varZeta}{1}{\ell},\eta,\point$.
	
	As in~\cite[Remark 75]{BinyaminiNovikov2019}, one may check that the cellular maps constructed in the proof of the \s CPT are translates in the last variable.
	Furthermore, they are prepared in the next-to-last variable --- during the proof we first apply the \s CPT in the base of the cell, yielding a translate in the next-to-last coordinate; We then possibly precompose with a power map in the base (see~\cite[Section 2.6]{BinyaminiNovikov2019} and \Cref{rem: composing prepared maps}).
	Hence, we have
	\begin{equation}
		g_j(\initial{\varZeta}{1}{\ell},\eta,\point)=(\dots,\pm\eta^{q_j}+\phi_j(\initial{\varZeta}{1}{\ell}),(g_j)_{\ell+2}(\initial{\varZeta}{1}{\ell},\eta))
	\end{equation}
	for some holomorphic function $\phi_j$ and $q_j\in\zIntegers_{\neq 0}$.
	We note that if $q_j\neq 1$, then the last coordinate of $\cell_j$ is not of type $\disc$.
	Thus, for all $q_j$, we either have $q_j \eta^{q_j -1}\neq 0$ or that the last coordinate of $\cell_j$ is of type $\point$.
	
	Let $\widetilde g_j=(\initial{(g_j)}{1}{\ell},(g_j)_{\ell+2}):\hExt{\cell_j}{\sigma}\to\hExt*{\cell\odot\cell[F]}{\rho}$.
	By definition of $Z$, we have that the last coordinate $\widetilde f \comp \widetilde g_j$ of $f\comp \widetilde g_j$ is equal to $\pm\eta^{q_j}+\phi_j(\initial{\varZeta}{1}{\ell})$.
	
	If the last coordinate of $\cell_j$ is of type $\point$, then $\widetilde g_j$ is clearly a cellular map, since $g_j$ is cellular.
	Assuming otherwise and noting that none of the coordinates $\initial{(g_j)}{1}{\ell}$ depend on $\eta$, we have
	\begin{equation}
		\pm q_j \eta^{q_{j}-1}
		=
		\pdv{\eta}(\widetilde f \comp \widetilde g_j)
		=
		\pdv{\varZ_{\ell+1}}\widetilde{f}
		\cdot
		\pdv{\eta}(g_j)_{\ell+2}.
	\end{equation}
	Therefore $\pdv{\eta}(g_j)_{\ell+2}\neq0$ and so $\widetilde g_j$ is a cellular map.
	
	It remains to check  that $\{\widetilde g_j:\hExt{\cell_j}{\sigma}\to\hExt*{\cell\odot\cell[F]}{\rho}\}$ is a cellular cover.
	Let $\initial{\varZ}{1}{\ell+1}\in\cell\odot\cell[F]$ and let $w=\widetilde{f}(\initial{\varZ}{1}{\ell+1})$.
	There is some $j$ such that ${(\initial{\varZ}{1}{\ell},w,\varZ_{\ell+1})\in g_j(\cell_j\odot\cell[F]_j)}$ and hence also ${\initial{\varZ}{1}{\ell+1}\in \widetilde g_j(\cell_j)}$, which finishes the proof.
\end{proof}

\appendix
\section{The real setting}
\label{sec: real cells}

We recall the notion of a \emph{real complex cell} and of \emph{real cellular maps} (see~\cite[Section 2.2.2]{BinyaminiNovikov2019}).
\begin{definition}
	The \emph{real part} $\rReal\cell$ of a complex cell $\cell$ of length $\ell$ is the intersection of $\cell$ with $\rReal^\ell$.
	The \emph{positive real part} $\rRealPos\cell$ of $\cell$ is similarly obtained by intersecting each coordinate of $\cell$ which is not of type $\point$ with the positive real numbers $\rRealPos=(0,\infty)$.
	A holomorphic function $f:\cell\to\cComplex$ is called \emph{real} if it is real valued on $\rReal\cell$.
	Finally, a complex cell $\cell$ is called \emph{real} if all radii defining its components are real holomorphic functions over their corresponding bases.
\end{definition}
We note that a real complex cell of length $\ell$ is a subset of $\cComplex^\ell$ and not of $\rReal^\ell$.

\begin{definition}
	Let $\cell$ be a complex cell and let $Z\subset\cell$.
	We will call $Z$ \emph{symmetric} if it is invariant under the conjugation $\varZ\mapsto\conjugate{\varZ}$.
\end{definition}
Thus real complex cells and the zero-sets of real holomorphic maps defined on real complex cells are symmetric.

\begin{definition}
	\label{def: real map and real cover}
	Let $\ext{\cell}{\delta}$ be a real complex cell.
	A cellular map $f:\cell\to\cComplex^\ell$ is \emph{real} if each component of $f$ is a real holomorphic function.
	A finite collection $\{f_j:\ext{\cell_j}{\delta'}\to\ext{\cell}{\delta}\}$ of real cellular maps is a \emph{real cellular cover} if $\rReal\cell\subset\cup_j (f_j(\rRealPos\cell_j))$.
\end{definition}

As in~\cite{BinyaminiNovikov2019}, the statements of the sharp refinement theorem, the \s CPT and the \s CPrT (\Cref{thm:sharp refinement,thm:cpt,thm:sharp preparation}, respectively) all have real versions, where we replace all complex cells, cellular maps and cellular covers in the hypotheses and conclusions by their real counterparts (and in the case of the \s CPT, assume that the sets $Z_i$ are symmetric).

The small necessary modifications to the proofs are the same as in~\cite{BinyaminiNovikov2019}.
We note that in \Cref{sec: removing an epsilon net} of the proof of the \s CPT, when intersecting the sets $Z_i$ with graphs of the constant functions $p\in\epsilon\zIntegers^2\subset\cComplex$, we may group these graphs by conjugate pairs in order to ensure that the resulting intersections and their projections are symmetric.

\begin{remark}
	In his master's thesis~\cite{Shankar2025}, Shankar proves a version of the real CPT in the algebraic setting where, in addition to the desired size and complexity bounds on the resulting cellular covers, the images of the positive real parts of the covering cells are disjoint.
	We refer to such covers as \emph{disjoint real cellular covers}.

	Though not necessary for this paper, this improvement is important in applications, for example to preparation theorems as in~\cite{CarmonAnalyticallyGenerated,BinyaminiCarmonNovikovLE}.
	
	The methods of~\cite{Shankar2025} extend to the \so-minimal setting, subject to to the following modifications.
	First, one must use arguments as in \Cref{sec: step 1 proof} to treat certain discriminant sets which appear in the proofs.
	Second, in all (inductive) applications of the \s CPT during the proofs, we include $\{z_i=\pm1\}$ and $\{z_i=\pm 2\}$ among the collection of hypersurfaces $\{Z_j\}$.
	This ensures that when splitting $\cComplex$-fibers according to \Cref{rem: complex basic fiber} (for example, in the proof of the \s CPrT), we may choose a subcollection of the resulting cellular covering maps which together form a disjoint real cellular cover.
	
	One may thus take the covers obtained by the real \s CPT to be disjoint real cellular covers.
\end{remark}

\bibliographystyle{abbrv}
\bibliography{references}
\end{document}